\documentclass[11pt]{article}
\usepackage{amsfonts}
\usepackage{amssymb,amsfonts,amsmath,amsthm,cite}
\usepackage{enumerate}
\usepackage{ytableau}
\usepackage{hyperref}

\usepackage{graphicx}
\usepackage{tikz}
\usetikzlibrary{calc,decorations.pathreplacing}
\usetikzlibrary{arrows,shapes,positioning}
\usetikzlibrary{decorations.markings}

\tikzset{middlearrow/.style={
        decoration={markings,
            mark= at position 0.5 with {\arrow{#1}} ,
        },
        postaction={decorate}
    }
}

\allowdisplaybreaks[4]

\theoremstyle{plain}
\newtheorem{theorem}{Theorem}[section]
\newtheorem{lemma}[theorem]{Lemma}

\theoremstyle{definition}
\newtheorem{definition}[theorem]{Definition}

\newtheorem{conjecture}[theorem]{Conjecture}

\theoremstyle{remark}
\newtheorem{remark}[theorem]{Remark}

\makeatletter \@addtoreset{equation}{section}

\begin{document}

\begin{center}
{\Large {\bf Partition Analysis and the Little Göllnitz Identites}}

\vskip 6mm

{\small Runqiao Li \\[2mm]

Department of Mathematics\\
University of Texas Rio Grande Valley\\
Edinburg, TX 78541, USA\\[3mm]

runqiao.li@utrgv.edu or runqaoli@outlook.com\\[2mm]}
\end{center}

\noindent {\bf Abstract.} This work follows the spirit of Andrews' series of papers on Partition Analysis. In $2011$, Savage and Sills found new sum sides for the little Göllnitz identities and provided their partition interpretations. It turns out that similar companions exist for a mod $8$ partition identity due to Andrews. In this work, we use MacMahon's Partition Analysis to study partitions related to these identities. We find refined generating functions for them, where we keep track of the size of each part. Finally, by considering the alternating sum and Schmidt weight, we show the application of these refined functions in the study of partition statistics.

\noindent \textbf{Keywords:} Partition Analysis, little Göllnitz identities, generating functions, Schmidt-type identites.

\noindent \textbf{AMS Classification: 05A17, 11P84, 11Y99}

\section{Introduction}

A partition is a finite weakly decreasing sequence of positive integers $\lambda=(\lambda_1,\lambda_2,\lambda_3\,\ldots,\lambda_{\ell})$. Each integer $\lambda_i$ is called a part of $\lambda$, and the number of parts is called the length of $\lambda$, denoted by $\ell(\lambda)$. We call $\lambda_{1},\lambda_3,\lambda_5\ldots$ the odd-indexed parts and $\lambda_{2},\lambda_4,\lambda_6\ldots$ the even-indexed parts. The weight of a partition is defined as the sum of its parts, denoted by $|\lambda|=\lambda_1+\lambda_2+\cdots\lambda_{\ell}$. We say $\lambda$ is a partition of $n$ if $|\lambda|=n$. The set of partitions is denoted by $\mathcal{P}$. A partition is called strict partition if the sequence is strictly decreasing.
 
Throughout this paper we will apply the $q$-pochhammer symbol defined by
$$(a;q)_{0}:=1,\quad(a;q)_n:=\prod_{i=0}^{n-1}(1-aq^i)\quad\text{and}\quad
(a;q)_{\infty}:=\lim_{n\to\infty}(a;q)_{n},
$$
where the infinite product is convergent for $|q|<1$. We also adopt the abbreviation
$$(a_1,a_2,\ldots,a_k;q)_{n}:=(a_1;q)_n(a_2;q)_n\cdots(a_k;q)_n$$
and
$$(a_1,a_2,\ldots,a_k;q)_{\infty}:=(a_1;q)_{\infty}(a_2;q)_{\infty}\cdots(a_k;q)_{\infty}.$$
The study of generating functions played an important role in the theory of partitions. It connected partitions with basic hypergeometry series. Let $\mathcal{G}_1$ be the set of partitions satisfying $\lambda_i-\lambda_{i+1}\geq2\ \text{and}\ \lambda_i-\lambda_{i+1}\geq3\ \text{if $\lambda_i$ is odd.}$ And let $\mathcal{G}_2$ be the set of partitions in $\mathcal{G}_1$ with the smallest part being at least $2$. Then the following two identities are known as the little Göllnitz identities~\cite{Göllnitz}, and we present them with the partition interpretations of their sum sides.
\begin{theorem}[Göllnitz]
The first little Göllnitz identity is given by
\begin{equation}\label{eq:QseriesG1}
\sum_{\lambda\in\mathcal{G}_1}q^{|\lambda|}=\sum_{n=0}^{\infty}\frac{q^{n^2+n}(-q^{-1};q^2)_{n}}{(q^2;q^2)_{n}}=\frac{1}{(q,q^5,q^6;q^8)_{\infty}}.
\end{equation}
And The Second little Göllnitz identity is given by
\begin{equation}\label{eq:QseriesG2}
\sum_{\lambda\in\mathcal{G}_2}q^{|\lambda|}=\sum_{n=0}^{\infty}\frac{q^{n^2+n}(-q;q^2)_{n}}{(q^2;q^2)_{n}}=\frac{1}{(q^2,q^3,q^7;q^8)_{\infty}}.
\end{equation}
\end{theorem}
In \cite{SavageSills} Savage and Sills found companions for \eqref{eq:QseriesG1} and \eqref{eq:QseriesG2}. They provided new sum sides together with the partition interpretations. We call them  the new little Göllnitz identites. Let $\mathcal{G}'_1$ be the set of strict partitions with even-indexed parts being even, while let $\mathcal{G}'_2$ be the set of strict partitions with odd-indexed parts being even.
\begin{theorem}[Savage and Sills]
The first new little Göllnitz identity is given by
\begin{equation}\label{eq:QseriesG1'}
\sum_{\lambda\in\mathcal{G}'_1}q^{|\lambda|}=\sum_{n=0}^{\infty}\frac{q^{\binom{2n}{2}}(-q;q^4)_{n}}{(q^2;q^2)_{2n}}=\frac{1}{(q,q^5,q^6;q^8)_{\infty}}.   
\end{equation}
And the Second new little Göllnitz identity is given by
\begin{equation}\label{eq:QseriesG2'}
\sum_{\lambda\in\mathcal{G}'_2}q^{|\lambda|}=\sum_{n=0}^{\infty}\frac{q^{\binom{2n+1}{2}}(-q^{-1};q^4)_{n}}{(q^2;q^2)_{2n}}=\frac{1}{(q^2,q^3,q^7;q^8)_{\infty}}.
\end{equation}
\end{theorem}

In 1967, Andrews~ \cite{AndrewsSchur} studied a partition set similar to $\mathcal{G}_2$ and found its generating function expressed as an infinite product. let $\mathcal{P}_2$ be the set of partitions $\lambda=(\lambda_1,\lambda_2,\ldots,\lambda_{\ell})$ such that $\lambda_{\ell}\geq2$ and $\lambda_{i}-\lambda_{i+1}\geq3$ if $\lambda_{i}$ is odd. In~\cite{AndrewsSIP}, he also provided a series form generating function and proved the following partition identity with modulus $8$.
\begin{theorem}[Andrews' mod $8$ identity]
The generating function for $\mathcal{P}_2$ satisfies
\begin{equation}\label{eq:QseriesP2}
\sum_{\lambda\in\mathcal{P}_2}q^{|\lambda|}=1+\sum_{n=1}^{\infty}\frac{(-q^3;q^4)_{n-1}q^{2n}(1+q^{2n-1})}{(q^2;q^2)_{n}}=\frac{1}{(q^2,q^3,q^4,q^{7},q^{8};q^{8})_{\infty}}.
\end{equation}
\end{theorem}
Equation \eqref{eq:QseriesP2} can be viewed as an analogue of \eqref{eq:QseriesG2}. So, it is natural to ask for analogues for \eqref{eq:QseriesG1}, \eqref{eq:QseriesG1'} and \eqref{eq:QseriesG2'}. Let $\mathcal{P}_1$ be the set of partitions such that $\lambda_{i}-\lambda_{i+1}\geq3$ if $\lambda_{i+1}$ is odd, we will have the following companion of \eqref{eq:QseriesP2}.
\begin{theorem}
The generating function for $\mathcal{P}_1$ satisfies
\begin{equation}\label{eq:QseriesP1}
\sum_{\lambda\in\mathcal{P}_1}q^{|\lambda|}=1+\sum_{n=1}^{\infty}\frac{(-q;q^4)_{n-1}q^{2n}(1+q^{2n-3})}{(q^2;q^2)_{n}}=\frac{1}{(q,q^4,q^5,q^6,q^8;q^8)_{\infty}}.    
\end{equation}
\end{theorem}
Next we present new sum sides for \eqref{eq:QseriesP1} and \eqref{eq:QseriesP2}. Let $\mathcal{P}_1'$ be the sets of ordinary partitions with even-indexed parts being even, and let $\mathcal{P}_2'$ be the set of ordinary partitions with odd-indexed parts being even. Then, the following mod $8$ identities hold.
\begin{theorem}
The generating functions for $\mathcal{P}_1'$ and $\mathcal{P}_2'$ satisfy
\begin{equation}\label{eq:QseriesP1'}
\sum_{\lambda\in\mathcal{P}_1'}q^{|\lambda|}=\sum_{n=0}^{\infty}\frac{q^{4n}(-q^{-3};q^4)_n}{(q^2;q^2)_{2n}}+\sum_{n=0}^{\infty}\frac{q^{4n+2}(-q;q^4)_n}{(q^2;q^2)_{2n+1}}=\frac{1}{(q,q^4,q^5,q^6,q^8;q^8)_{\infty}}  
\end{equation}
and
\begin{equation}\label{eq:QseriesP2'}
\sum_{\lambda\in\mathcal{P}_2'}q^{|\lambda|}=\sum_{n=0}^{\infty}\frac{q^{4n}(-q^{-1};q^4)_n}{(q^2;q^2)_{2n}}+\sum_{n=0}^{\infty}\frac{q^{4n+2}(-q^3;q^4)_{n}}{(q^2;q^2)_{2n+1}}=\frac{1}{(q^2,q^3,q^4,q^7,q^8;q^8)_{\infty}}.    
\end{equation}
\end{theorem}
Since 1998, Andrews has written a series of papers devoted to Partition Analysis. Some of them are joint works with other scholars. The main advantage of this method is that we can assign different variables to keep track of each part in partitions. Let $X=\{x_1,x_2,x_3,\ldots\}$ be a collection of countably many variables. For a partition $\lambda=(\lambda_1,\lambda_2,\ldots,\lambda_{\ell})$, we define the refined weight as 
$$x^{\lambda}:=x_1^{\lambda_1}x_2^{\lambda_2}x_3^{\lambda_3}\cdots x_{\ell}^{\lambda_{\ell}}.$$
For the sake of compactness, we also define
$$X_i:=\left\{\begin{array}{cc}
   x_1x_2\cdots x_i  & \text{for}\quad i\geq1,\\
   1 & \text{otherwise.} 
\end{array}\right.$$
Recently, Andrews and Paule \cite{AndrewsPauleSchmidt,AndrewsPauleNN,AndrewsPauleParity} applied Partition Analysis on various partition sets and presented their generating functions with respect to the refined weight. Inspired by their work, we shall take the same approach to study the partitions associated with all the mod $8$ identities mentioned above. Our main goal is the following. For partitions related to the little Göllnitz identities and the new little Göllnitz identites, we shall prove the refined series form.  
\begin{theorem}\label{thm:main1}
For partitions in $\mathcal{G}_1$, $\mathcal{G}_2$, $\mathcal{G}_1'$, and $\mathcal{G}_2'$, the following refined generating functions hold.
\begin{equation}\label{SummaryG1}
\sum_{\lambda\in\mathcal{G}_1}x^{\lambda}=1+\sum_{n=1}^{\infty}\frac{X_1^{2}X_2^{2}\cdots X_n^2(1+X_1)(1+X_1X_2)\cdots(1+X_{n-2}X_{n-1})(1+X_{n-1}/X_n)}{(1-X_1^2)(1-X_2^2)\cdots(1-X_n^2)},
\end{equation}
\begin{equation}\label{SummaryG2}
\sum_{\lambda\in\mathcal{G}_2}x^{\lambda}=1+\sum_{n=1}^{\infty}\frac{X_1^2X_2^2\cdots X_{n}^{2}(1+X_1)(1+X_1X_2)\cdots(1+X_{n-1}X_n)}{(1-X_1^2)(1-X_2^2)\cdots(1-X_n^2)},    
\end{equation}
\begin{equation}\label{SummaryG1'}
\sum_{\lambda\in\mathcal{G}_{1}'}x^{\lambda}=1+\sum_{n=1}^{\infty}\frac{X_1X_2\cdots X_{2n-1}(1+X_1)(1+X_2X_3)\cdots(1+X_{2n-2}X_{2n-1})}{(1-X_1^{2})(1-X_2^2)\cdots(1-X_{2n}^{2})}, 
\end{equation}
\begin{equation}\label{SummaryG2'}
\sum_{\lambda\in\mathcal{G}_{2}'}x^{\lambda}=1+\sum_{n=1}^{\infty}\frac{X_1X_2\cdots X_{2n}(1+X_1X_2)(1+X_3X_4)\cdots(1+X_{2n-3}X_{2n-2})(1-X_{2n-1}/X_{2n})}{(1-X_1^{2})(1-X_2^{2})\cdots(1-X_{2n}^{2})}. 
\end{equation}
\end{theorem}
Moreover, for Andrews' mod $8$ identity and their companions, we shall prove the following refined identities.
\begin{theorem}\label{thm:main2}
For partitions in $\mathcal{P}_1$, $\mathcal{P}_2$, $\mathcal{P}_1'$, and $\mathcal{P}_2'$, the following refined partition identities hold.
\begin{equation}\label{SummaryP1}
\begin{split}
\sum_{\lambda\in\mathcal{P}_1}x^{\lambda}=&1+\sum_{n=1}^{\infty}\frac{X_n^2(1+X_1)(1+X_1^3X_2)\cdots(1+X_{n-2}^3X_{n-1})(1+X_{n-1}^3/X_n)}{(1-X_1^2)(1-X_2^2)\cdots(1-X_n^2)}\\
=&\frac{1+X_1}{1-X_1^2}\cdot\frac{1+X_1^3X_2}{1-X_2^2}\cdot\frac{1+X_2^3X_3}{1-X_3^2}\cdot\frac{1+X_3^3X_4}{1-X_4^2}\cdots,
\end{split}    
\end{equation}
\begin{equation}\label{SummaryP2}
\begin{split}
\sum_{\lambda\in\mathcal{P}_2}x^{\lambda}=&1+\sum_{n=1}^{\infty}\frac{X_n^2(1+X_1^3)(1+X_1X_2^3)\cdots(1+X_{n-2}X_{n-1}^3)(1+X_{n-1}X_n)}{(1-X_1^2)(1-X_2^2)\cdots(1-X_n^2)}\\
=&\frac{1+X_1^3}{1-X_1^2}\cdot\frac{1+X_1X_2^3}{1-X_2^2}\cdot\frac{1+X_2X_3^3}{1-X_3^2}\cdot\frac{1+X_3X_4^3}{1-X_4^2}\cdots,    
\end{split}
\end{equation}
\begin{equation}\label{SummaryP1'}
\begin{split}
\sum_{\lambda\in\mathcal{P}'_1}x^{\lambda}=&1+\sum_{n=1}^{\infty}\frac{X_{2n}^{2}(1+X_1)(1+X_2X_3)\cdots(1+X_{2n-4}X_{2n-3})(1+X_{2n-2}X_{2n-1}/X_{2n}^2)}{(1-X_1^2)(1-X_2^2)\cdots(1-X_{2n}^{2})}\\
&+\sum_{n=0}^{\infty}\frac{X_{2n+1}^{2}(1+X_1)(1+X_2X_3)\cdots(1+X_{2n-2}X_{2n-1})}{(1-X_1^2)(1-X_2^2)\cdots(1-X_{2n+1}^{2})}\\
=&\frac{1+X_1}{1-X_1^2}\cdot\frac{1}{1-X_2^2}\cdot\frac{1+X_2X_3}{1-X_3^2}\cdot\frac{1}{1-X_4^2}\cdots,    
\end{split}
\end{equation}
\begin{equation}\label{SummaryP2'}
\begin{split}
\sum_{\lambda\in\mathcal{P}'_2}x^{\lambda}=&1+\sum_{n=1}^{\infty}\frac{X_{2n}^2(1+X_1X_2)(1+X_3X_4)\cdots(1+X_{2n-3}X_{2n-2})(1+X_{2n-1}/X_{2n})}{(1-X_1^2)(1-X_2^2)\cdots(1-X_{2n}^2)}\\
&+\sum_{n=0}^{\infty}\frac{X_{2n+1}^2(1+X_1X_2)(1+X_3X_4)\cdots(1+X_{2n-1}X_{2n})}{(1-X_1^2)(1-X_2^2)\cdots(1-X_{2n+1}^2)}\\
=&\frac{1}{1-X_1^2}\cdot\frac{1+X_1X_2}{1-X_2^2}\cdot\frac{1}{1-X_3^2}\cdot\frac{1+X_3X_4}{1-X_4^2}\cdots.  
\end{split}
\end{equation}
\end{theorem}
\begin{remark}
The product side of \eqref{SummaryP2'} was proved in \cite{AndrewsPauleParity}.   
\end{remark}
The refined functions give us more flexibility. For instance, if we let $x_i\to q$ for all $i$, Theorem~\ref{thm:main1} will recover the sum sides of \eqref{eq:QseriesG1}, \eqref{eq:QseriesG2}, \eqref{eq:QseriesG1'} and \eqref{eq:QseriesG2'}, and \eqref{eq:QseriesP1}, \eqref{eq:QseriesP2}, \eqref{eq:QseriesP1'} and \eqref{eq:QseriesP2'} will be corollaries of Theorem~\ref{thm:main2}. Meanwhile, with some more subtle substitution, we will be able to keep track of various statistics on integer partitions, and we shall see more examples on this later.

The rest of the paper is organized as follows. In Section~\ref{sec:PA}, we introduce MacMahon's Partition Analysis. We will explain the process of this method and provide all the necessary elimination rules. In Section~\ref{sec:LG} and \ref{sec:NLG}, we treat the partitions related to the little Göllnitz identities and the new little Göllnitz identities, respectively. And in Section~\ref{sec:P} and \ref{sec:P'}, we treat the partitions related to Andrews' mod $8$ identity and its companions. Section~\ref{sec:Application} is devoted to the applications of our main results. We shall define the alternating sum and Schmidt weight of a partition, and show how they can be reflected from the refined generating functions. Finally, in Section~\ref{sec:Conclusion}, we discuss some problems inspired by this project that might worth a further investigation.

\section{Partition Analysis}\label{sec:PA}

Macmahon's Partition Analysis relies on the Omega operator defined as follows.
\begin{definition}
The Omega operator $\Omega_{\geq}$ is given by
$$\underset{\geq}{\Omega}\sum_{s_1=-\infty}^{\infty}\cdots\sum_{s_r=-\infty}^{\infty}A_{s_1,\ldots,s_r}\lambda_1^{s_1}\cdots\lambda_r^{s_r}:=\sum_{s_1=0}^{\infty}\cdots\sum_{s_r=0}^{\infty}A_{s_1,\ldots,s_r},$$
where the domain of the $A_{s_1\ldots,s_r}$ is the field of rational functions over $\mathbb{C}$
in several complex variables and the $\lambda_i$'s are restricted to a neighborhood
of the circle $|\lambda_i|=1$. In addition, the $A_{s_1\ldots,s_r}$ are required to be such that
any of the series involved is absolutely convergent within the domain of the
definition of $A_{s_1\ldots,s_r}$.
\end{definition}

Loosely speaking, when we apply the Omega operator to a power series of $\lambda_i$'s, we delete terms who have at least one negative exponent and then set $\lambda_i\to1$ for all $i$. This was first introduced by P. A. MacMahon in his book Combinatorial Analysis\cite{MacMahon}. He designed this method to prove his conjecture on plane partitions. And it also turned out to be a powerful tool with wide application in partition theory.

When we apply the Partition Analysis, the first step is to translate the restriction of partitions into inequalities. Let's take the set of all partitions as an example. Since the $\lambda_i$'s are reversed as variables to be eliminated in Omega operator, we will use $\alpha=(a_1,a_2,\ldots,a_{\ell})$ to denote partitions instead. The generating function of partitions with at most $n$ parts is
\begin{equation}\label{eq:ExampleStep1}
\sum_{\substack{\alpha\in\mathcal{P}\\\ell(\alpha)\leq n}}x^{\alpha}=\sum_{a_1\geq a_2\geq\cdots\geq a_n\geq0}x^{a_1}x_2^{a_2}\cdots x_{n}^{a_n}.
\end{equation}
Note that the condition $a_i\geq a_{i+1}$ can be written as $a_{i}-a_{i+1}\geq0$. So, by the Omega operator, \eqref{eq:ExampleStep1} implies
\begin{equation}\label{eq:ExampleStep2}
\sum_{\substack{\alpha\in\mathcal{P}\\\ell(\alpha)\leq n}}x^{\alpha}=\underset{\geq}{\Omega}\sum_{a_1,a_2,\ldots,a_n\geq0}x_1^{a_1}x_2^{a_{2}}\cdots x_n^{a_n}\lambda_1^{a_1-a_2}\lambda_2^{a_2-a_3}\cdots\lambda_{n-1}^{a_{n-1}-a_n}\lambda_n^{a_n}.
\end{equation}
In \eqref{eq:ExampleStep2}, the advantage is that the range of the summations are independent. So, we can add them up separately. This leads to
\begin{equation}\label{eq:ExampleCrude}
\sum_{\substack{\alpha\in\mathcal{P}\\\ell(\alpha)\leq n}}x^{\alpha}=\underset{\geq}{\Omega}\frac{1}{(1-x_1\lambda_1)(1-x_2\lambda_2/\lambda_1)\cdots(1-x_{n}\lambda_{n}/\lambda_{n-1})}.
\end{equation}
And this is what we called the `crude form' of the generating function. Next we need to use proper elimination rules to cancel the $\lambda_i$'s. For \eqref{eq:ExampleCrude}, we need the following.
\begin{lemma}
For any integer $A\geq0$ and $|x|, |y| < 1$,
\begin{equation}\label{eq:EliminationExample}
\underset{\geq}{\Omega}\frac{\lambda^{-A}}{(1-x\lambda)(1-\frac{y}{\lambda})}=\frac{x^{A}}{(1-x)(1-xy)}.
\end{equation}
\end{lemma}

\begin{proof}
By the definition of the Omega operator,
\begin{align*}
\underset{\geq}{\Omega}\frac{\lambda^{-A}}{(1-x\lambda)(1-\frac{y}{\lambda})}=&\underset{\geq}{\Omega}\sum_{n=0}^{\infty}\sum_{m=0}^{\infty}x^{n}y^{m}\lambda^{n-m-A}=\sum_{m=0}^{\infty}\sum_{n=m+A}^{\infty}x^{n}y^{m}=\frac{x^{A}}{(1-x)(1-xy)}.
\end{align*}
This completes the proof.
\end{proof}

\begin{remark}
In~\cite{AndrewsLectureHall}, Andrews listed a number of frequently used rules, and \eqref{eq:EliminationExample} is one of them. However, we present the proof here because it demonstrates the mechanical of the Omega operator. It also set an example for proving other elimination rules.
\end{remark}
Now, we apply \eqref{eq:EliminationExample} to cancel the $\lambda_i$'s in \eqref{eq:ExampleCrude} one by one, starting from $\lambda_1$.
\begin{align*}
\sum_{\substack{\alpha\in\mathcal{P}\\\ell(\alpha)\leq n}}x^{\alpha}=&\underset{\geq}{\Omega}\frac{1}{(1-x_1\lambda_1)(1-x_2\lambda_2/\lambda_1)\cdots(1-x_{n}\lambda_{n}/\lambda_{n-1})}\\
\intertext{(by \eqref{eq:EliminationExample} with $A\to0$, $x\to x_1$ and $y\to x_2\lambda_2$)}
=&\underset{\geq}{\Omega}\frac{1}{(1-x_1)(1-x_1x_2\lambda_2)(1-x_3\lambda_3/\lambda_2)\cdots(1-x_{n}\lambda_{n}/\lambda_{n-1})}\\
\intertext{(by
\eqref{eq:EliminationExample} with $A\to0$, $x\to x_1x_2$ and $y\to x_3\lambda_3$)}
=&\underset{\geq}{\Omega}\frac{1}{(1-x_1)(1-x_1x_2)(1-x_1x_2x_3\lambda_3)(1-x_4\lambda_4/\lambda_3)\cdots(1-x_{n}\lambda_{n}/\lambda_{n-1})}\\
&\vdots\\
=&\frac{1}{(1-x_1)(1-x_2)\cdots(1-x_1x_2\cdots x_n)}\\
=&\prod_{i=1}^{n}\frac{1}{(1-X_i)}.
\end{align*}
So we get the refined generating function. To find the expression for all partitions in $\mathcal{P}$, we simply need to let $n\to\infty$.

Next we present a complete list of elimination rules that will be used in the sequel.
\begin{lemma}\label{EliminationRule}
The following elimination rules hold for the Omega operator.
\begin{equation}\label{EliminationRule0}
\underset{\geq}{\Omega}\frac{1+\frac{z}{\lambda}}{(1-x\lambda^2)(1-\frac{y}{\lambda^2})\lambda^2}=\frac{x(1+xz)}{(1-x)(1-xy)},
\end{equation}
\begin{equation}\label{EliminationRule01}
\underset{\geq}{\Omega}\frac{1+z\lambda}{(1-x\lambda^2)(1-\frac{y}{\lambda^2})\lambda^2}=\frac{x(1+z)}{(1-x)(1-xy)},
\end{equation}
\begin{equation}\label{EliminationRule1}
\underset{\geq}{\Omega}\frac{1+z\lambda}{\lambda(1-x\lambda^2)(1-\frac{y}{\lambda^2})}=\frac{x+z}{(1-x)(1-xy)},  
\end{equation}
\begin{equation}\label{EliminationRule2}
\underset{\geq}{\Omega}\frac{1}{\lambda(1-x\lambda)(1-\frac{y}{\lambda^{2}})}=\frac{x}{(1-x)(1-x^2y)},    
\end{equation}
\begin{equation}\label{EliminationRule3}
\underset{\geq}{\Omega}\frac{1}{\lambda(1-x\lambda^2)(1-\frac{y}{\lambda})}=\frac{x(1+y)}{(1-x)(1-xy^2)},
\end{equation}
\begin{equation}\label{EliminationRule4}
\underset{\geq}{\Omega}\frac{1+z\lambda}{(1-x\lambda^2)(1-\frac{y}{\lambda^2})}=\frac{1+z}{(1-x)(1-xy)},
\end{equation}
\begin{equation}\label{EliminationRule5}
\underset{\geq}{\Omega}\frac{1}{(1-x\lambda)(1-\frac{y}{\lambda^{2}})}=\frac{1}{(1-x)(1-x^2y)},    
\end{equation}
\begin{equation}\label{EliminationRule6}
\underset{\geq}{\Omega}\frac{1}{(1-x\lambda^2)(1-\frac{y}{\lambda})}=\frac{1+xy}{(1-x)(1-xy^2)},
\end{equation}
\begin{equation}\label{EliminationRule7}
\underset{\geq}{\Omega}\frac{1+\frac{z}{\lambda^4}}{(1-x\lambda)(1-\frac{y}{\lambda^2})}=\frac{1+x^4z}{(1-x)(1-x^2y)},   
\end{equation}
\begin{equation}\label{EliminationRule8}
\underset{\geq}{\Omega}\frac{(1+w\lambda)(1+\frac{z}{\lambda^4})}{(1-x\lambda^2)(1-\frac{y}{\lambda^2})}=\frac{(1+w)(1+x^2z)}{(1-x)(1-xy)},   
\end{equation}
\begin{equation}\label{EliminationRule9}
\underset{\geq}{\Omega}\frac{(1+\frac{w}{\lambda^2})(1+\frac{z}{\lambda})}{(1-x\lambda^2)(1-\frac{y}{\lambda^2})}=\frac{(1+wx)(1+xz)}{(1-x)(1-xy)}
\end{equation}
\end{lemma}
\begin{proof}
Here we present a complete proof for \eqref{EliminationRule0}.
\begin{align*}
\underset{\geq}{\Omega}\frac{1+\frac{z}{\lambda}}{(1-x\lambda^2)(1-\frac{y}{\lambda^2})\lambda^2}=&\underset{\geq}{\Omega}\frac{1}{(1-x\lambda^2)(1-\frac{y}{\lambda^2})\lambda^2}+\underset{\geq}{\Omega}\frac{z}{(1-x\lambda^2)(1-\frac{y}{\lambda^2})\lambda^3}\\
=&
\underset{\geq}{\Omega}\sum_{n=0}^{\infty}\sum_{m=0}^{\infty}x^ny^m\lambda^{2n-2m-2}+z\underset{\geq}{\Omega}\sum_{n=0}^{\infty}\sum_{m=0}^{\infty}x^ny^m\lambda^{2n-2m-3}\\
=&
\sum_{n=1}^{\infty}\sum_{m=0}^{n-1}x^ny^m+z\sum_{n=2}^{\infty}\sum_{m=0}^{n-2}x^ny^m\\
=&
\frac{x}{(1-x)(1-xy)}+\frac{x^2z}{(1-x)(1-xy)}\\
=&
\frac{x(1+xz)}{(1-x)(1-xy)}.
\end{align*}
The rest of the identities can be proved in the same way.
\end{proof}

\begin{remark}
The elimination of $\lambda$'s can also be performed using the Mathematica package Omega \cite{OmegaPackage,OmegaWeb}, which can also test the rules we present above. Some of them might have been included in \cite{AndrewsLectureHall} as well. 
\end{remark}
In addition, we will also need the following indicator for odd numbers.
$$\chi(n):=\left\{\begin{array}{cc}
    1 & \text{if $n$ is odd,} \\
    0 & \text{if $n$ is even.}
\end{array}\right.$$
\begin{lemma}
For any integer $k\geq0$,
\begin{equation}\label{LemmaChi}
\sum_{n\geq0}x^{n}\lambda^{n-k\chi(n)}=\frac{1+x\lambda^{1-k}}{1-x^2\lambda^2}.
\end{equation}   
\end{lemma}
\begin{proof}
It is straight forward to check
$$\sum_{n\geq0}x^{n}\lambda^{n-k\chi(n)}=\sum_{n=0}^{\infty}x^{2n}\lambda^{2n}+\sum_{n=0}^{\infty}x^{2n+1}\lambda^{2n+1-k}=\frac{1}{1-x^2\lambda^2}+\frac{x\lambda^{1-k}}{1-x^{2}\lambda^2}=\frac{1+x\lambda^{1-k}}{1-x^{2}\lambda^2},$$
which finishes the proof.
\end{proof}

\begin{remark}
The special case of equation \eqref{LemmaChi} with $k=2$  was first introduced in \cite{AndrewsPauleParity} to deal with the Göllnitz-Gordon identities. Here we provide a generalization, which gives more flexibility to handle the gap conditions. We will need this to deduce the crude form for $\mathcal{G}_1$, $\mathcal{G}_2$, $\mathcal{P}_1$ and $\mathcal{P}_2$.    
\end{remark}

\section{The little Göllnitz identities}\label{sec:LG}

In this section, we apply Partition Analysis to prove refined generating functions for partitions associated with the little Göllnitz identities.

\begin{theorem}\label{LambdaXG1}
The refined generating function for partitions in $\mathcal{G}_1$ is given by
$$\sum_{\lambda\in\mathcal{G}_1}x^{\lambda}=1+\sum_{n=1}^{\infty}\frac{X_1^{2}X_2^{2}\cdots X_n^2(1+X_1)(1+X_1X_2)\cdots(1+X_{n-2}X_{n-1})(1+X_{n-1}/X_n)}{(1-X_1^2)(1-X_2^2)\cdots(1-X_n^2)}.$$
\end{theorem}

\begin{proof}
We start by considering the partitions in $\mathcal{G}_1$ with length $n$, and the crude form is
\begin{align*}
\sum_{\substack{\alpha\in\mathcal{G}_1\\\ell(\alpha)=n}}x^{\alpha}
=&\underset{\geq}{\Omega}\sum_{a_1,\ldots,a_n\geq0}^{\infty}x_1^{a_1}\cdots x_{n}^{a_n}\lambda_1^{a_1-a_2-2-\chi(a_1)}\cdots\lambda_{n-1}^{a_{n-1}-a_{n}-2-\chi(a_{n-1})}\lambda_n^{a_n-1}\\
=&\underset{\geq}{\Omega}\frac{1+x_1}{1-x_1^2\lambda_1^2}\prod_{i=2}^{n-1}\left(1+\frac{x_i}{\lambda_{i-1}}\right)\prod_{i=2}^{n-1}\left(1-\frac{x_i^{2}\lambda_i^{2}}{\lambda_{i-1}^2}\right)^{-1}\cdot\frac{1}{1-x_n\lambda_n/\lambda_{n-1}}\cdot\frac{1}{\lambda_1^2\cdots\lambda_{n-1}^2\lambda_n}\\
=&\underset{\geq}{\Omega}\frac{1+x_1}{1-x_1^2\lambda_1^2}\prod_{i=2}^{n-1}\left(1+\frac{x_i}{\lambda_{i-1}}\right)\prod_{i=2}^{n}\left(1-\frac{x_i^{2}\lambda_i^{2}}{\lambda_{i-1}^2}\right)^{-1}\cdot\frac{1+x_n\lambda_n/\lambda_{n-1}}{\lambda_1^2\cdots\lambda_{n-1}^2\lambda_n}.
\end{align*}
Applying \eqref{EliminationRule0} in Lemma \ref{EliminationRule} to eliminate the $\lambda$'s, starting with $\lambda_1$,
\begin{align*}
&\underset{\geq}{\Omega}\frac{1+x_1}{1-x_1^2\lambda_1^2}\prod_{i=2}^{n-1}\left(1+\frac{x_i}{\lambda_{i-1}}\right)\prod_{i=2}^{n}\left(1-\frac{x_i^{2}\lambda_i^{2}}{\lambda_{i-1}^2}\right)^{-1}\cdot\frac{1+x_n\lambda_n/\lambda_{n-1}}{\lambda_1^2\cdots\lambda_{n-1}^2\lambda_n}\\
=&\frac{X_1^2(1+X_1)(1+X_1X_2)}{1-X_1^2}\underset{\geq}{\Omega}\frac{1}{1-X_2^2\lambda_2^2}\prod_{i=3}^{n}\left(1+\frac{x_i}{\lambda_{i-1}}\right)\prod_{i=3}^{n}\left(1-\frac{x_i^{2}\lambda_i^{2}}{\lambda_{i-1}^2}\right)\cdot\frac{1+x_n\lambda_n/\lambda_{n-1}}{\lambda_2^2\cdots\lambda_{n-1}^2\lambda_n}\\
&\vdots\\
=&\frac{X_1^{2}X_{2}^{2}\cdots X_{n-2}^2(1+X_1)(1+X_1X_2)\cdots(1+X_{n-2}X_{n-1})}{(1-X_1^2)(1-X_2^2)\cdots(1-X_{n-2}^2)}\\
&\times\underset{\geq}{\Omega}\frac{1+x_n\lambda_n/\lambda_{n-1}}{(1-X_{n-1}^2\lambda_{n-1}^2)(1-x_{2n}^2\lambda_{2n}^2/\lambda_{2n-1}^2)\lambda_{2n-1}^2\lambda_{2n}}\\
=&
\frac{X_1^{2}X_{2}^{2}\cdots X_{n-1}^2(1+X_1)(1+X_1X_2)\cdots(1+X_{n-2}X_{n-1})}{(1-X_1^2)(1-X_2^2)\cdots(1-X_{n-1}^2)}\underset{\geq}{\Omega}\frac{1+X_{n-1}X_n\lambda_n}{(1-X_{n}^2\lambda_{n}^2)\lambda_n}\\
=&\frac{X_1^{2}X_2^{2}\cdots X_n^2(1+X_1)(1+X_1X_2)\cdots(1+X_{n-2}X_{n-1})(1+X_{n-1}/X_n)}{(1-X_1^2)(1-X_2^2)\cdots(1-X_n^2)},
\end{align*}
where in the last step we applied \eqref{EliminationRule1} with $y=0$. We have established the generating function for partitions with length being equal to $n$, it remains to add up all the nonnegative integers $n$. So we finish the proof.
\end{proof}

Similarly, we can get a series form for the refined generating function of partitions in $\mathcal{G}_2$. 
\begin{theorem}\label{LambdaXG2}
The refined generating function for partitions in $\mathcal{G}_2$ is given by
$$\sum_{\lambda\in\mathcal{G}_2}x^{\lambda}=1+\sum_{n=1}^{\infty}\frac{X_1^2X_2^2\cdots X_{n}^{2}(1+X_1)(1+X_1X_2)\cdots(1+X_{n-1}X_n)}{(1-X_1^2)(1-X_2^2)\cdots(1-X_n^2)}.$$
\end{theorem}

\begin{proof}
The crude form for the generating function of partitions in $\mathcal{G}_2$ with length $n$ is given by
\begin{align*}
\sum_{\substack{\alpha\in\mathcal{G}_2\\\ell(\alpha)=n}}x^{\alpha}=&\underset{\geq}{\Omega}\sum_{a_1,\ldots,a_n\geq0}^{\infty}x_1^{a_1}\cdots x_{n}^{a_n}\lambda_1^{a_1-a_2-2-\chi(a_1)}\cdots\lambda_{n-1}^{a_{n-1}-a_{n}-2-\chi(a_{n-1})}\lambda_n^{a_n-2}\\
=&\underset{\geq}{\Omega}\frac{1+x_1}{1-x_1^2\lambda_1^2}\prod_{i=2}^{n-1}\left(1+\frac{x_i}{\lambda_{i-1}}\right)\prod_{i=2}^{n-1}\left(1-\frac{x_i^{2}\lambda_i^{2}}{\lambda_{i-1}^2}\right)^{-1}\cdot\frac{1}{1-x_n\lambda_n/\lambda_{n-1}}\cdot\frac{1}{\lambda_1^2\cdots\lambda_{n-1}^2\lambda_n^2}\\
=&\underset{\geq}{\Omega}\frac{1+x_1}{1-x_1^2\lambda_1^2}\prod_{i=2}^{n-1}\left(1+\frac{x_i}{\lambda_{i-1}}\right)\prod_{i=2}^{n}\left(1-\frac{x_i^{2}\lambda_i^{2}}{\lambda_{i-1}^2}\right)^{-1}\cdot\frac{1+x_n\lambda_n/\lambda_{n-1}}{\lambda_1^2\cdots\lambda_{n-1}^2\lambda_n^2}.
\end{align*}
We shall use \eqref{EliminationRule0} to eliminate the $\lambda$'s, starting from $\lambda_1$.
\begin{align*}
&\underset{\geq}{\Omega}\frac{1+x_1}{1-x_1^2\lambda_1^2}\prod_{i=2}^{n-1}\left(1+\frac{x_i}{\lambda_{i-1}}\right)\prod_{i=2}^{n}\left(1-\frac{x_i^{2}\lambda_i^{2}}{\lambda_{i-1}^2}\right)^{-1}\cdot\frac{1+x_n\lambda_n/\lambda_{n-1}}{\lambda_1^2\cdots\lambda_{n-1}^2\lambda_n^2}\\
=&\frac{X_1^2(1+X_1)(1+X_1X_2)}{1-X_1^2}\underset{\geq}{\Omega}\frac{1}{1-X_2^2\lambda_2^2}\prod_{i=3}^{n}\left(1+\frac{x_i}{\lambda_{i-1}}\right)\prod_{i=3}^{n}\left(1-\frac{x_i^{2}\lambda_i^{2}}{\lambda_{i-1}^2}\right)\cdot\frac{1+x_n\lambda_n/\lambda_{n-1}}{\lambda_2^2\cdots\lambda_{n-1}^2\lambda_n^2}\\
&\vdots\\
=&\frac{X_1^{2}X_{2}^{2}\cdots X_{n-2}^2(1+X_1)(1+X_1X_2)\cdots(1+X_{n-2}X_{n-1})}{(1-X_1^2)(1-X_2^2)\cdots(1-X_{n-2}^2)}\\
&\times\underset{\geq}{\Omega}\frac{1+x_n\lambda_n/\lambda_{n-1}}{(1-X_{n-1}^2\lambda_{n-1}^2)(1-x_{2n}^2\lambda_{2n}^2/\lambda_{2n-1}^2)\lambda_{2n-1}^2\lambda_{2n}^2}\\
=&
\frac{X_1^{2}X_{2}^{2}\cdots X_{n-1}^2(1+X_1)(1+X_1X_2)\cdots(1+X_{n-2}X_{n-1})}{(1-X_1^2)(1-X_2^2)\cdots(1-X_{n-1}^2)}\underset{\geq}{\Omega}\frac{1+X_{n-1}X_n\lambda_n}{(1-X_{n}^2\lambda_{n}^2)\lambda_n^2}\\
=&\frac{X_1^{2}X_2^{2}\cdots X_n^2(1+X_1)(1+X_1X_2)\cdots(1+X_{n-2}X_{n-1})(1+X_{n-1}X_n)}{(1-X_1^2)(1-X_2^2)\cdots(1-X_n^2)},
\end{align*}
where in the last step we applied \eqref{EliminationRule01} with $y=0$. By adding up all the nonnegative integers $n$, we finish the proof.
\end{proof}

\section{The new little Göllnitz identities}\label{sec:NLG}

In this section we treat the partitions related to the new little Göllnitz identities.
\begin{theorem}\label{LambdaXG1'}
The refined generating function for partitions with even-indexed parts being even is given by
$$\sum_{\lambda\in\mathcal{G}_{1}'}x^{\lambda}=1+\sum_{n=1}^{\infty}\frac{X_1X_2\cdots X_{2n-1}(1+X_1)(1+X_2X_3)\cdots(1+X_{2n-2}X_{2n-1})}{(1-X_1^{2})(1-X_2^2)\cdots(1-X_{2n}^{2})}.$$
\end{theorem}

\begin{proof}
We start by considering the partitions in $\mathcal{G}_1'$ with length $2n-1$ or $2n$. The crude form is given by
\begin{align*}
\sum_{\substack{\alpha\in\mathcal{G}_1'\\2n-1\leq\ell(\alpha)\leq2n}}x^{\alpha}=&\underset{\geq}{\Omega}\sum_{a_1,a_2,\ldots,a_{2n}\geq0}x_1^{a_1}x_2^{2a_{2}}\cdots x_{2n-1}^{a_{2n-1}}x_{2n}^{2a_{2n}}\lambda_1^{a_1-2a_2-1}\lambda_{2}^{2a_2-a_3-1}\cdots\\
&\times\lambda_{2n-1}^{a_{2n-1}-2a_{2n}-1}\lambda_{2n}^{2a_{2n}}\\
=&
\underset{\geq}{\Omega}\frac{\lambda_1^{-1}\cdots\lambda_{2n-1}^{-1}}{1-x_1\lambda_1}\prod_{i=2}^{n}\left(1-\frac{x_{2i-1}\lambda_{2i-1}}{\lambda_{2i-2}}\right)^{-1}\prod_{i=1}^{n}\left(1-\frac{x_{2i}^{2}\lambda_{2i}^{2}}{\lambda_{2i-1}^{2}}\right)^{-1}.
\end{align*}
We  will need \eqref{EliminationRule1},\eqref{EliminationRule2} and \eqref{EliminationRule3} from Lemma \ref{EliminationRule} to cancel the $\lambda$'s as follows.
\begin{align*}
&\underset{\geq}{\Omega}\frac{\lambda_1^{-1}\cdots\lambda_{2n-1}^{-1}}{1-x_1\lambda_1}\prod_{i=2}^{n}\left(1-\frac{x_{2i-1}\lambda_{2i-1}}{\lambda_{2i-2}}\right)^{-1}\prod_{i=1}^{n}\left(1-\frac{x_{2i}^{2}\lambda_{2i}^{2}}{\lambda_{2i-1}^{2}}\right)^{-1}\\
=&\frac{X_1}{1-X_1}\underset{\geq}{\Omega}\frac{\lambda_2^{-1}\cdots\lambda_{2n-1}^{-1}}{1-X_2^2\lambda_2^2}\prod_{i=2}^{n}\left(1-\frac{x_{2i-1}\lambda_{2i-1}}{\lambda_{2i-2}}\right)^{-1}\prod_{i=2}^{n}\left(1-\frac{x_{2i}^{2}\lambda_{2i}^{2}}{\lambda_{2i-1}^{2}}\right)^{-1}\\
=&\frac{X_1X_2^2(1+X_1)}{(1-X_1^2)(1-X_2^2)}\underset{\geq}{\Omega}\frac{\lambda_3^{-1}\cdots\lambda_{2n-1}^{-1}(1+x_3\lambda_3)}{1-X_3^2\lambda_3^2}\prod_{i=3}^{n}\left(1-\frac{x_{2i-1}\lambda_{2i-1}}{\lambda_{2i-2}}\right)^{-1}\prod_{i=2}^{n}\left(1-\frac{x_{2i}^{2}\lambda_{2i}^{2}}{\lambda_{2i-1}^{2}}\right)^{-1}\\
=&\frac{X_1X_2X_3(1+X_1)(1+X_2X_3)}{(1-X_1^2)(1-X_2^2)(1-X_3^2)}\underset{\geq}{\Omega}\frac{\lambda_4^{-1}\cdots\lambda_{2n-1}^{-1}}{1-X_4^2\lambda_4^2}\prod_{i=3}^{n}\left(1-\frac{x_{2i-1}\lambda_{2i-1}}{\lambda_{2i-2}}\right)^{-1}\prod_{i=3}^{n}\left(1-\frac{x_{2i}^{2}\lambda_{2i}^{2}}{\lambda_{2i-1}^{2}}\right)^{-1}\\
&\vdots\\
=&\frac{X_1X_2\cdots X_{2n-1}(1+X_1)(1+X_2X_3)\cdots(1+X_{2n-2}X_{2n-1})}{(1-X_1^{2})(1-X_2^2)\cdots(1-X_{2n-1}^{2})}\underset{\geq}{\Omega}\frac{1}{1-X_{2n}\lambda_{2n}^2}\\
=&\frac{X_1X_2\cdots X_{2n-1}(1+X_1)(1+X_2X_3)\cdots(1+X_{2n-2}X_{2n-1})}{(1-X_1^{2})(1-X_2^2)\cdots(1-X_{2n}^{2})},
\end{align*}
which gives the desired expression for those partitions with length $2n-1$ or $2n$. By adding up all the positive integers $n$ and including the empty partition, we finish the proof.
\end{proof}

\begin{theorem}\label{LambdaXG2'}
The refined generating function for partitions with odd-indexed part be even is given by
$$\sum_{\lambda\in\mathcal{G}_{2}'}x^{\lambda}=1+\sum_{n=1}^{\infty}\frac{X_1X_2\cdots X_{2n}(1+X_1X_2)(1+X_3X_4)\cdots(1+X_{2n-1}/X_{2n})}{(1-X_1^{2})(1-X_2^{2})\cdots(1-X_{2n}^{2})}.$$   
\end{theorem}
\begin{proof}
By considering the partitions in $\mathcal{G}_2'$ with length $2n-1$ or $2n$, we have
\begin{align*}
\sum_{\substack{\alpha\in\mathcal{G}_2'\\2n-1\leq\ell(\alpha)\leq2n}}x^{\alpha}=&\underset{\geq}{\Omega}\sum_{a_1,\ldots,a_{2n}\geq0}x_1^{2a_{1}}x_2^{a_2}\cdots x_{2n-1}^{2a_{2n-1}}x_{2n}^{a_{2n}}\lambda_{1}^{2a_{1}-a_2-1}\lambda_2^{a_2-2a_3-1}\cdots\\
&\times\lambda_{2n-2}^{a_{2n-2}-2a_{2n-1}-1}\lambda_{2n-1}^{2a_{2n-1}-a_{2n}-1}\lambda_{2n}^{a_{2n}}\\
=&\underset{\geq}{\Omega}\frac{\lambda_1^{-1}\cdots\lambda_{2n}^{-1}}{1-x_1^2\lambda_1^2}\prod_{i=2}^{n}\left(1-\frac{x_{2i-1}^2\lambda_{2i-1}^{2}}{\lambda_{2i-2}^2}\right)^{-1}\prod_{i=1}^{n}\left(1-\frac{x_{2i}\lambda_{2i}}{\lambda_{2i-1}}\right)^{-1}
\end{align*}
as the crude form. Next, we eliminate the $\lambda$'s to get the desired expression, and for that we need \eqref{EliminationRule1}, \eqref{EliminationRule3} and \eqref{EliminationRule4}.
\begin{align*}
&\underset{\geq}{\Omega}\frac{\lambda_1^{-1}\cdots\lambda_{2n-1}^{-1}}{1-x_1^2\lambda_1^2}\prod_{i=2}^{n}\left(1-\frac{x_{2i-1}^2\lambda_{2i-1}^{2}}{\lambda_{2i-2}^2}\right)^{-1}\prod_{i=1}^{n}\left(1-\frac{x_{2i}\lambda_{2i}}{\lambda_{2i-1}}\right)^{-1}\\
=&\frac{X_1^2}{1-X_1^2}\underset{\geq}{\Omega}\frac{\lambda_2^{-1}\cdots\lambda_{2n-1}^{-1}(1+x_2\lambda_2)}{1-X_2^2\lambda_2^2}\prod_{i=2}^{n}\left(1-\frac{x_{2i-1}^2\lambda_{2i-1}^{2}}{\lambda_{2i-2}^2}\right)^{-1}\prod_{i=2}^{n}\left(1-\frac{x_{2i}\lambda_{2i}}{\lambda_{2i-1}}\right)^{-1}\\
=&\frac{X_1X_2(1+X_1X_2)}{(1-X_1^2)(1-X_2^2)}\underset{\geq}{\Omega}\frac{\lambda_3^{-1}\cdots\lambda_{2n-1}^{-1}}{1-X_3^2\lambda_3^2}\prod_{i=3}^{n}\left(1-\frac{x_{2i-1}^2\lambda_{2i-1}^{2}}{\lambda_{2i-2}^2}\right)^{-1}\prod_{i=2}^{n}\left(1-\frac{x_{2i}\lambda_{2i}}{\lambda_{2i-1}}\right)^{-1}\\
&\vdots\\
=&\frac{X_1X_2\cdots X_{2n-1}^{2}(1+X_1X_2)(1+X_3X_4)\cdots(1+X_{2n-3}X_{2n-2})}{(1-X_1^2)(1-X_2^2)\cdots(1-X_{2n-1}^{2})}\underset{\geq}{\Omega}\frac{1+x_{2n}\lambda_{2n}}{1-X_{2n}^2\lambda_{2n}^2}\\
=&\frac{X_1X_2\cdots X_{2n-1}^2(1+X_1X_2)(1+X_3X_4)\cdots(1+X_{2n-3}X_{2n-2})(1+x_{2n})}{(1-X_1^2)(1-X_2^2)\cdots(1-X_{2n}^{2})}\\
=&\frac{X_1X_2\cdots X_{2n}(1+X_1X_2)(1+X_3X_4)\cdots(1+X_{2n-3}X_{2n-2})(1+X_{2n-1}/X_{2n})}{(1-X_1^2)(1-X_2^2)\cdots(1-X_{2n}^{2})},
\end{align*}
which generates the partitions in $\mathcal{G}_2'$ with length $2n-1$ or $2n$. By adding up all the positive integers $n$, we have
$$\sum_{\lambda\in\mathcal{G}_{2}'}x^{\lambda}=1+\sum_{n=1}^{\infty}\frac{X_1X_2\cdots X_{2n}(1+X_1X_2)(1+X_3X_4)\cdots(1+X_{2n-1}/X_{2n})}{(1-X_1^{2})(1-X_2^{2})\cdots(1-X_{2n}^{2})}$$
as desired. So we finish the proof.
\end{proof}

\section{Andrews' mod $8$ partition identity}\label{sec:P}
In this section we give the refined partition identities related to partitions in $\mathcal{P}_1$ and $\mathcal{P}_2$, respectively.
\begin{theorem}\label{LambdaXP1}
For partitions in $\mathcal{P}_1$,
\begin{align*}
\sum_{\lambda\in\mathcal{P}_1}x^{\lambda}=&1+\sum_{n=1}^{\infty}\frac{X_n^2(1+X_1)(1+X_1^3X_2)\cdots(1+X_{n-2}^3X_{n-1})(1+X_{n-1}^3/X_n)}{(1-X_1^2)(1-X_2^2)\cdots(1-X_n^2)}\\
=&\frac{1+X_1}{1-X_1^2}\cdot\frac{1+X_1^3X_2}{1-X_2^2}\cdot\frac{1+X_2^3X_3}{1-X_3^2}\cdot\frac{1+X_3^3X_4}{1-X_4^2}\cdots.
\end{align*}
\end{theorem}

\begin{proof}
We start by considering the series form, which is based on the counting of partitions in $\mathcal{P}_1$ with length $n$. Recall that we defined $\chi(n)$ as the indicator function for odd integers, thus the crude form is given by
\begin{align*}
\sum_{\substack{\lambda\in\mathcal{P}_1\\\ell(\lambda)=n}}x^{\lambda}=&\underset{\geq}{\Omega}\sum_{a_1,a_2\ldots,a_n\geq0}x_1^{a_1}x_2^{a_2}\cdots x_n^{a_n}\lambda_{1}^{a_1-a_2-3\chi(a_2)}\cdots\lambda_{n-1}^{a_{n-1}-a_{n}-3\chi(a_n)}\lambda_{n}^{a_n-1}\\
=&\underset{\geq}{\Omega}\frac{1}{1-x_1\lambda_1}\prod_{i=2}^{n}\left(\frac{1+x_i\lambda_i/\lambda_{i-1}^4}{1-x_i^{2}\lambda_i^{2}/\lambda_{i-1}^2}\right)\cdot\frac{1}{\lambda_n}.
\end{align*}
Now, by \eqref{EliminationRule7} and \eqref{EliminationRule8} in Lemma \ref{EliminationRule},
\begin{align*}
&\underset{\geq}{\Omega}\frac{1}{1-x_1\lambda_1}\prod_{i=2}^{n}\left(\frac{1+x_i\lambda_i/\lambda_{i-1}^4}{1-x_i^{2}\lambda_i^{2}/\lambda_{i-1}^2}\right)\cdot\frac{1}{\lambda_n}\\
=&\frac{1}{1-X_1}\underset{\geq}{\Omega}\frac{1+X_1^3X_2\lambda_2}{1-X_2^2\lambda_2^2}\prod_{i=3}^{n}\left(\frac{1+x_i\lambda_i/\lambda_{i-1}^4}{1-x_i^{2}\lambda_i^{2}/\lambda_{i-1}^2}\right)\cdot\frac{1}{\lambda_n}\\
=&\frac{(1+X_1)(1+X_1^3X_2)}{(1-X_1^2)(1-X_2^2)}\underset{\geq}{\Omega}\frac{1+X_2^3X_3\lambda_3}{1-X_3^2\lambda_3^2}\prod_{i=4}^{n}\left(\frac{1+x_i\lambda_i/\lambda_{i-1}^4}{1-x_i^{2}\lambda_i^{2}/\lambda_{i-1}^2}\right)\cdot\frac{1}{\lambda_n}\\
&\vdots\\
=&\frac{(1+X_1)(1+X_1^3X_2)\cdots(1+X_{n-2}^2X_{n-1})}{(1-X_1^2)(1-X_2^2)\cdots(1-X_{n-1}^2)}\underset{\geq}{\Omega}\frac{1+X_{n-1}^3X_n\lambda_n}{(1-X_n^2\lambda_n^2)\lambda_n}\\
=&\frac{X_n^2(1+X_1)(1+X_1^3X_2)\cdots(1+X_{n-2}^2X_{n-1})(1+X_{n-1}^3/X_n)}{(1-X_1^2)(1-X_2^2)\cdots(1-X_{n}^2)},
\end{align*}
hence we have shown that
$$\sum_{\substack{\lambda\in\mathcal{P}_1\\\ell(\lambda)=n}}x^{\lambda}=\frac{X_n^2(1+X_1)(1+X_1^3X_2)\cdots(1+X_{n-2}^2X_{n-1})(1+X_{n-1}^3/X_n)}{(1-X_1^2)(1-X_2^2)\cdots(1-X_{n}^2)}.$$
By adding up all the nonnegative integers $n$, we get the series form. As for the product form, it suffices to consider the partitions in $\mathcal{P}_1$ with length bounded by $n$, and the crude form is
\begin{align*}
\sum_{\substack{\alpha\in\mathcal{P}_1\\\ell(\alpha)\leq n}}x^{\alpha}=&\underset{\geq}{\Omega}\sum_{a_1,\dots,a_n\geq0}x_1^{a_1}x_2^{a_2}\cdots x_n^{a_n}\lambda_{1}^{a_1-a_2-3\chi(a_2)}\cdots\lambda_{n-1}^{a_{n-1}-a_{n}-3\chi(a_n)}\lambda_{n}^{a_n}\\
=&\underset{\geq}{\Omega}\frac{1}{1-x_1\lambda_1}\prod_{i=2}^{n}\left(\frac{1+x_i\lambda_i/\lambda_{i-1}^4}{1-x_i^{2}\lambda_i^{2}/\lambda_{i-1}^2}\right).
\end{align*}
By the same process, we eliminate the $\lambda$'s and get
\begin{align*}
&\underset{\geq}{\Omega}\frac{1}{1-x_1\lambda_1}\prod_{i=2}^{n}\left(\frac{1+x_i\lambda_i/\lambda_{i-1}^4}{1-x_i^{2}\lambda_i^{2}/\lambda_{i-1}^2}\right)\\
=&\frac{1}{1-X_1}\underset{\geq}{\Omega}\frac{1+X_1^3X_2\lambda_2}{1-X_2^2\lambda_2^2}\prod_{i=3}^{n}\left(\frac{1+x_i\lambda_i/\lambda_{i-1}^4}{1-x_i^{2}\lambda_i^{2}/\lambda_{i-1}^2}\right)\\
=&\frac{(1+X_1)(1+X_1^3X_2)}{(1-X_1^2)(1-X_2^2)}\underset{\geq}{\Omega}\frac{1+X_2^3X_3\lambda_3}{1-X_3^2\lambda_3^2}\prod_{i=4}^{n}\left(\frac{1+x_i\lambda_i/\lambda_{i-1}^4}{1-x_i^{2}\lambda_i^{2}/\lambda_{i-1}^2}\right)\\
&\vdots\\
=&\frac{(1+X_1)(1+X_1^3X_2)\cdots(1+X_{n-2}^2X_{n-1})}{(1-X_1^2)(1-X_2^2)\cdots(1-X_{n-1}^2)}\underset{\geq}{\Omega}\frac{1+X_{n-1}^3X_n\lambda_n}{1-X_n^2\lambda_n^2}\\
=&\frac{(1+X_1)(1+X_1^3X_2)\cdots(1+X_{n-2}^2X_{n-1})(1+X_{n-1}^3X_n)}{(1-X_1^2)(1-X_2^2)\cdots(1-X_{n}^2)}.
\end{align*}
So we have
$$\sum_{\substack{\alpha\in\mathcal{P}_1\\\ell(\alpha)\leq n}}x^{\alpha}=\frac{(1+X_1)(1+X_1^3X_2)\cdots(1+X_{n-2}^2X_{n-1})(1+X_{n-1}^3X_n)}{(1-X_1^2)(1-X_2^2)\cdots(1-X_{n}^2)}.$$
Let $n\to\infty$, we have the desired product form, hence we finish the proof.
\end{proof}

\begin{theorem}\label{LambdaXP2}
For partitions in $\mathcal{P}_2$,
\begin{align*}
\sum_{\lambda\in\mathcal{P}_2}x^{\lambda}=&1+\sum_{n=1}^{\infty}\frac{X_n^2(1+X_1^3)(1+X_1X_2^3)\cdots(1+X_{n-2}X_{n-1}^3)(1+X_{n-1}X_n)}{(1-X_1^2)(1-X_2^2)\cdots(1-X_n^2)}\\
=&\frac{1+X_1^3}{1-X_1^2}\cdot\frac{1+X_1X_2^3}{1-X_2^2}\cdot\frac{1+X_2X_3^3}{1-X_3^2}\cdot\frac{1+X_3X_4^3}{1-X_4^2}\cdots.
\end{align*}    
\end{theorem}
\begin{proof}
We shall follow the same process as the proof of Theorem \ref{LambdaXP1}. For partitions in $\mathcal{P}_2$ with length $n$, the crude form is
\begin{align*}
\sum_{\substack{\alpha\in\mathcal{P}_2\\\ell(\alpha)=n}}x^{\alpha}=&\underset{\geq}{\Omega}\sum_{a_1,a_2\ldots,a_n\geq0}x_1^{a_1}x_2^{a_2}\cdots x_n^{a_n}\lambda_{1}^{a_1-a_2-3\chi(a_1)}\cdots\lambda_{n-1}^{a_{n-1}-a_{n}-3\chi(a_{n-1})}\lambda_{n}^{a_n-2}\\
=&\underset{\geq}{\Omega}\frac{1+x_1/\lambda_1^2}{1-x_1^2\lambda_1^2}\prod_{i=2}^{n-1}\left(\frac{1+x_i/\lambda_{i-1}\lambda_{i}^2}{1-x_i^{2}\lambda_i^{2}/\lambda_{i-1}^2}\right)\cdot\frac{1}{(1-x_n\lambda_n/\lambda_{n-1})\lambda_n^2}\\
=&\underset{\geq}{\Omega}\frac{1+x_1/\lambda_1^2}{1-x_1^2\lambda_1^2}\prod_{i=2}^{n-1}\left(\frac{1+x_i/\lambda_{i-1}\lambda_{i}^2}{1-x_i^{2}\lambda_i^{2}/\lambda_{i-1}^2}\right)\cdot\frac{1+x_n\lambda_n/\lambda_{n-1}}{(1-x_n^2\lambda_n^2/\lambda_{n-1}^2)\lambda_n^2}.
\end{align*}
By \eqref{EliminationRule9} in Lemma \ref{EliminationRule},
\begin{align*}
&\underset{\geq}{\Omega}\frac{1+x_1/\lambda_1^2}{1-x_1^2\lambda_1^2}\prod_{i=2}^{n-1}\left(\frac{1+x_i/\lambda_{i-1}\lambda_{i}^2}{1-x_i^{2}\lambda_i^{2}/\lambda_{i-1}^2}\right)\cdot\frac{1+x_n\lambda_n/\lambda_{n-1}}{(1-x_n^2\lambda_n^2/\lambda_{n-1}^2)\lambda_n^2}\\
=&\frac{1+X_1^3}{1-X_1^2}\underset{\geq}{\Omega}\frac{1+X_1X_2/\lambda_2^2}{1-X_2^2\lambda_2^2}\prod_{i=3}^{n-1}\left(\frac{1+x_i/\lambda_{i-1}\lambda_{i}^2}{1-x_i^{2}\lambda_i^{2}/\lambda_{i-1}^2}\right)\cdot\frac{1+x_n\lambda_n/\lambda_{n-1}}{(1-x_n^2\lambda_n^2/\lambda_{n-1}^2)\lambda_n^2}\\
=&\frac{(1+X_1^3)(1+X_1X_2^3)}{(1-X_1^2)(1-X_2^{2})}\underset{\geq}{\Omega}\frac{1+X_2X_3/\lambda_3^2}{1-X_3^2\lambda_3^2}\prod_{i=4}^{n-1}\left(\frac{1+x_i/\lambda_{i-1}\lambda_{i}^2}{1-x_i^{2}\lambda_i^{2}/\lambda_{i-1}^2}\right)\cdot\frac{1+x_n\lambda_n/\lambda_{n-1}}{(1-x_n^2\lambda_n^2/\lambda_{n-1}^2)\lambda_n^2}\\
&\vdots\\
=&\frac{(1+X_1^3)(1+X_1X_2^3)\cdots(1+X_{n-2}X_{n-1}^3)}{(1-X_1^2)(1-X_2^{2})\cdots(1-X_{n-1}^2)}\underset{\geq}{\Omega}\frac{1+X_{n-1}X_{n}\lambda_n}{(1-X_n^2\lambda_n^2)\lambda_n^2}\\
=&\frac{X_n^2(1+X_1^3)(1+X_1X_2^3)\cdots(1+X_{n-2}X_{n-1}^3)(1+X_{n-1}X_n)}{(1-X_1^2)(1-X_2^{2})\cdots(1-X_{n}^2)}.
\end{align*}
We have shown that
$$\sum_{\substack{\alpha\in\mathcal{P}_2\\\ell(\alpha)=n}}x^{\alpha}=\frac{X_n^2(1+X_1^3)(1+X_1X_2^3)\cdots(1+X_{n-2}X_{n-1}^3)(1+X_{n-1}X_n)}{(1-X_1^2)(1-X_2^{2})\cdots(1-X_{n}^2)},$$
and by adding up all the nonnegative integers $n$, we get the series form. Next we consider such partitions with length bounded by $n$.
\begin{align*}
\sum_{\substack{\alpha\in\mathcal{P}_2\\\ell(\alpha)\leq n}}x^{\alpha}=&\underset{\geq}{\Omega}\sum_{a_1,a_2\ldots,a_n\geq0}x_1^{a_1}x_2^{a_2}\cdots x_n^{a_n}\lambda_{1}^{a_1-a_2-3\chi(a_1)}\cdots\lambda_{n-1}^{a_{n-1}-a_{n}-3\chi(a_{n-1})}\lambda_{n}^{a_n-3\chi(a_n)}\\
=&\underset{\geq}{\Omega}\frac{1+x_1/\lambda_1^2}{1-x_1^2\lambda_1^2}\prod_{i=2}^{n}\left(\frac{1+x_i/\lambda_{i-1}\lambda_{i}^2}{1-x_i^{2}\lambda_i^{2}/\lambda_{i-1}^2}\right)\\
=&\frac{1+X_1^3}{1-X_1^2}\underset{\geq}{\Omega}\frac{1+X_1X_2/\lambda_2^2}{1-X_2^2\lambda_2^2}\prod_{i=3}^{n}\left(\frac{1+x_i/\lambda_{i-1}\lambda_{i}^2}{1-x_i^{2}\lambda_i^{2}/\lambda_{i-1}^2}\right)\\
=&\frac{(1+X_1^3)(1+X_1X_2^3)}{(1-X_1^2)(1-X_2^{2})}\underset{\geq}{\Omega}\frac{1+X_2X_3/\lambda_3^2}{1-X_3^2\lambda_3^2}\prod_{i=4}^{n}\left(\frac{1+x_i/\lambda_{i-1}\lambda_{i}^2}{1-x_i^{2}\lambda_i^{2}/\lambda_{i-1}^2}\right)\\
&\vdots\\
=&\frac{(1+X_1^3)(1+X_1X_2^3)\cdots(1+X_{n-2}X_{n-1}^3)}{(1-X_1^2)(1-X_2^{2})\cdots(1-X_{n-1}^2)}\underset{\geq}{\Omega}\frac{1+X_{n-1}X_{n}/\lambda_n^2}{1-X_n^2\lambda_n^2}\\
=&\frac{(1+X_1^3)(1+X_1X_2^3)\cdots(1+X_{n-2}X_{n-1}^3)(1+X_{n-1}X_n^3)}{(1-X_1^2)(1-X_2^{2})\cdots(1-X_{n}^2)},
\end{align*}
where the elimination is by iterating \eqref{EliminationRule9}. Let $n\to\infty$, we get the product form as excepted. So we finish the proof.
\end{proof}

\section{Ordinary partitions with position parity}\label{sec:P'}

In this section we give the refined identities for ordinary partitions with restricted position parity, starting with even-indexed parts being even.
\begin{theorem}\label{LambdaXP1'}
For the partitions in $\mathcal{P}'_1$,
\begin{align*}
\sum_{\lambda\in\mathcal{P}'_1}x^{\lambda}=&1+\sum_{n=1}^{\infty}\frac{X_{2n}^{2}(1+X_1)(1+X_2X_3)\cdots(1+X_{2n-4}X_{2n-3})(1+X_{2n-2}X_{2n-1}/X_{2n}^2)}{(1-X_1^2)(1-X_2^2)\cdots(1-X_{2n}^{2})}\\
&+\sum_{n=0}^{\infty}\frac{X_{2n+1}^{2}(1+X_1)(1+X_2X_3)\cdots(1+X_{2n-2}X_{2n-1})}{(1-X_1^2)(1-X_2^2)\cdots(1-X_{2n+1}^{2})}\\
=&\frac{1+X_1}{1-X_1^2}\cdot\frac{1}{1-X_2^2}\cdot\frac{1+X_2X_3}{1-X_3^2}\cdot\frac{1}{1-X_4^2}\cdots.
\end{align*}
\end{theorem}
\begin{proof}
We will show both the series form and the product form for the generating functions. The series form relies on counting such partitions with a fixed length $2n$ or $2n+1$. We first consider those of even length. The crude form is
\begin{align*}
\sum_{\substack{\alpha\in\mathcal{P}_1'\\\ell(\lambda)=2n}}x^{\alpha}=&
\underset{\geq}{\Omega}\sum_{a_1,\ldots a_{2n}\geq0}x_1^{a_1}x_2^{2a_{2}}\cdots x_{2n-1}^{a_{2n-1}}x_{2n}^{2a_{2n}}\lambda_{1}^{a_1-2a_{2}}\lambda_{2}^{2a_2-a_3}\cdots\lambda_{2n-1}^{a_{2n-1}-2a_{2n}}\lambda_{2n}^{2a_{2n}-1}\\
=&\underset{\geq}{\Omega}\frac{1}{1-x_1\lambda_1}\prod_{i=2}^{n}\left(1-\frac{x_{2i-1}\lambda_{2i-1}}{\lambda_{2i-2}}\right)^{-1}\prod_{i=1}^{n}\left(1-\frac{x_{2i}^2\lambda_{2i}^2}{\lambda_{2i-1}^2}\right)^{-1}\cdot\frac{1}{\lambda_{2n}}.
\end{align*}
Now, by \eqref{EliminationRule4}, \eqref{EliminationRule5} and \eqref{EliminationRule6} from Lemma \ref{EliminationRule},
\begin{align*}
&\underset{\geq}{\Omega}\frac{1}{1-x_1\lambda_1}\prod_{i=2}^{n}\left(1-\frac{x_{2i-1}\lambda_{2i-1}}{\lambda_{2i-2}}\right)^{-1}\prod_{i=1}^{n}\left(1-\frac{x_{2i}^2\lambda_{2i}^2}{\lambda_{2i-1}^2}\right)^{-1}\cdot\frac{1}{\lambda_{2n}}\\
=&\frac{1}{1-X_1}\underset{\geq}{\Omega}\frac{1}{1-X_2^2\lambda_2^2}\prod_{i=2}^{n}\left(1-\frac{x_{2i-1}\lambda_{2i-1}}{\lambda_{2i-2}}\right)^{-1}\prod_{i=2}^{n}\left(1-\frac{x_{2i}^2\lambda_{2i}^2}{\lambda_{2i-1}^2}\right)^{-1}\cdot\frac{1}{\lambda_{2n}}\\
=&\frac{1+X_1}{(1-X_1^2)(1-X_2^2)}\underset{\geq}{\Omega}\frac{1+X_2X_3\lambda_{3}}{1-X_3^2\lambda_3^2}\prod_{i=3}^{n}\left(1-\frac{x_{2i-1}\lambda_{2i-1}}{\lambda_{2i-2}}\right)^{-1}\prod_{i=2}^{n}\left(1-\frac{x_{2i}^2\lambda_{2i}^2}{\lambda_{2i-1}^2}\right)^{-1}\cdot\frac{1}{\lambda_{2n}}\\
=&\frac{(1+X_1)(1+X_2X_3)}{(1-X_1^2)(1-X_2^2)(1-X_3^2)}\underset{\geq}{\Omega}\frac{1}{1-X_4^2\lambda_4^2}\prod_{i=3}^{n}\left(1-\frac{x_{2i-1}\lambda_{2i-1}}{\lambda_{2i-2}}\right)^{-1}\prod_{i=3}^{n}\left(1-\frac{x_{2i}^2\lambda_{2i}^2}{\lambda_{2i-1}^2}\right)^{-1}\cdot\frac{1}{\lambda_{2n}}\\
&\vdots\\
=&\frac{(1+X_1)(1+X_2X_3)\cdots(1+X_{2n-2}X_{2n-1})}{(1-X_1^2)(1-X_2^2)\cdots(1-X_{2n-1}^{2})}\underset{\geq}{\Omega}\frac{1}{(1-X_{2n}^2\lambda_{2n}^2)\lambda_{2n}}\\
=&\frac{X_{2n}^{2}(1+X_1)(1+X_2X_3)\cdots(1+X_{2n-2}X_{2n-1})}{(1-X_1^2)(1-X_2^2)\cdots(1-X_{2n}^{2})}.
\end{align*}
We have shown that
$$\sum_{\substack{\alpha\in\mathcal{P}_1'\\\ell(\alpha)=2n}}x^{\alpha}
=\frac{X_{2n}^{2}(1+X_1)(1+X_2X_3)\cdots(1+X_{2n-2}X_{2n-1})}{(1-X_1^2)(1-X_2^2)\cdots(1-X_{2n}^{2})},$$
with a similar argument, we get
$$\sum_{\substack{\alpha\in\mathcal{P}_1'\\\ell(\alpha)=2n+1}}x^{\alpha}
=\frac{X_{2n+1}^{2}(1+X_1)(1+X_2X_3)\cdots(1+X_{2n-2}X_{2n-1})(1-X_{2n}/X_{2n+1})}{(1-X_1^2)(1-X_2^2)\cdots(1-X_{2n+1}^{2})}.$$
Now, adding up all the nonnegative integers $n$, we have
\begin{align*}
\sum_{\lambda\in\mathcal{P}'_1}x^{\lambda}=&1+\sum_{n=1}^{\infty}\frac{X_{2n}^2(1+X_1)(1+X_2X_3)\cdots(1+X_{2n-2}X_{2n-1})}{(1-X_1^2)(1-X_2^2)\cdots(1-X_{2n}^2)}\\
&+\sum_{n=0}^{\infty}\frac{X_{2n+1}^2(1+X_1)(1+X_2X_3)\cdots(1+X_{2n-2}X_{2n-1})(1+X_{2n}/X_{2n+1})}{(1-X_1^2)(1-X_2^2)\cdots(1-X_{2n+1}^2)}\\
=&1+\sum_{n=1}^{\infty}\frac{X_{2n}^2(1+X_1)(1+X_2X_3)\cdots(1+X_{2n-2}X_{2n-1})}{(1-X_1^2)(1-X_2^2)\cdots(1-X_{2n}^2)}\\
&+\sum_{n=1}^{\infty}\frac{X_{2n-1}^2(1+X_1)(1+X_2X_3)\cdots(1+X_{2n-4}X_{2n-3})(1+X_{2n-2}/X_{2n-1})}{(1-X_1^2)(1-X_2^2)\cdots(1-X_{2n-1}^2)}\\
=&1+\sum_{n=1}^{\infty}\frac{X_{2n-1}^{2}(1+X_1)(1+X_2X_3)\cdots(1+X_{2n-4}X_{2n-3})}{(1-X_1^2)(1-X_2^2)\cdots(1-X_{2n-1}^{2})}\\
&\times\left(\frac{(1+X_{2n-2}X_{2n-1})X_{2n}^2/X_{2n-1}^2}{1-X_{2n}^2}+1+X_{2n-2}/X_{2n-1}\right)\\
=&1+\sum_{n=1}^{\infty}\frac{X_{2n-1}^{2}(1+X_1)(1+X_2X_3)\cdots(1+X_{2n-4}X_{2n-3})}{(1-X_1^2)(1-X_2^2)\cdots(1-X_{2n-1}^{2})}\\
&\times\frac{1-X_{2n}^2+X_{2n}^2/X_{2n-1}^2+X_{2n-2}/X_{2n-1}}{1-X_{2n}^2}\\
=&1+\sum_{n=1}^{\infty}\frac{X_{2n}^{2}(1+X_1)(1+X_2X_3)\cdots(1+X_{2n-4}X_{2n-3})(1+X_{2n-2}X_{2n-1}/X_{2n}^2)}{(1-X_1^2)(1-X_2^2)\cdots(1-X_{2n}^{2})}\\
&+\sum_{n=0}^{\infty}\frac{X_{2n+1}^{2}(1+X_1)(1+X_2X_3)\cdots(1+X_{2n-2}X_{2n-1})}{(1-X_1^2)(1-X_2^2)\cdots(1-X_{2n+1}^{2})},
\end{align*}
which gives the series form of the generating function as desired. Next, we treat the product form,  for which we consider the partitions in $\mathcal{P}_1'$ with length bounded by $2n$. The crude form is given by
\begin{align*}
\sum_{\substack{\alpha\in\mathcal{P}_1'\\\ell(\lambda)\leq2n}}x^{\alpha}=&
\underset{\geq}{\Omega}\sum_{a_1,\ldots a_{2n}\geq0}x_1^{a_1}x_2^{2a_{2}}\cdots x_{2n-1}^{a_{2n-1}}x_{2n}^{2a_{2n}}\lambda_{1}^{a_1-2a_{2}}\lambda_{2}^{2a_2-a_3}\cdots\lambda_{2n-1}^{a_{2n-1}-2a_{2n}}\lambda_{2n}^{2a_{2n}}\\
=&\underset{\geq}{\Omega}\frac{1}{1-x_1\lambda_1}\prod_{i=2}^{n}\left(1-\frac{x_{2i-1}\lambda_{2i-1}}{\lambda_{2i-2}}\right)^{-1}\prod_{i=1}^{n}\left(1-\frac{x_{2i}^2\lambda_{2i}^2}{\lambda_{2i-1}^2}\right)^{-1}.
\end{align*}
By eliminating the $\lambda$'s with \eqref{EliminationRule4}, \eqref{EliminationRule5} and \eqref{EliminationRule6},
\begin{align*}
&\underset{\geq}{\Omega}\frac{1}{1-x_1\lambda_1}\prod_{i=2}^{n}\left(1-\frac{x_{2i-1}\lambda_{2i-1}}{\lambda_{2i-2}}\right)^{-1}\prod_{i=1}^{n}\left(1-\frac{x_{2i}^2\lambda_{2i}^2}{\lambda_{2i-1}^2}\right)^{-1}\\
=&\frac{1}{1-X_1}\underset{\geq}{\Omega}\frac{1}{1-X_2^2\lambda_2^2}\prod_{i=2}^{n}\left(1-\frac{x_{2i-1}\lambda_{2i-1}}{\lambda_{2i-2}}\right)^{-1}\prod_{i=2}^{n}\left(1-\frac{x_{2i}^2\lambda_{2i}^2}{\lambda_{2i-1}^2}\right)^{-1}\\
=&\frac{1+X_1}{(1-X_1^2)(1-X_2^2)}\underset{\geq}{\Omega}\frac{1+X_2X_3\lambda_{3}}{1-X_3^2\lambda_3^2}\prod_{i=3}^{n}\left(1-\frac{x_{2i-1}\lambda_{2i-1}}{\lambda_{2i-2}}\right)^{-1}\prod_{i=2}^{n}\left(1-\frac{x_{2i}^2\lambda_{2i}^2}{\lambda_{2i-1}^2}\right)^{-1}\\
=&\frac{(1+X_1)(1+X_2X_3)}{(1-X_1^2)(1-X_2^2)(1-X_3^2)}\underset{\geq}{\Omega}\frac{1}{1-X_4^2\lambda_4^2}\prod_{i=3}^{n}\left(1-\frac{x_{2i-1}\lambda_{2i-1}}{\lambda_{2i-2}}\right)^{-1}\prod_{i=3}^{n}\left(1-\frac{x_{2i}^2\lambda_{2i}^2}{\lambda_{2i-1}^2}\right)^{-1}\\
&\vdots\\
=&\frac{(1+X_1)(1+X_2X_3)\cdots(1+X_{2n-2}X_{2n-1})}{(1-X_1^2)(1-X_2^2)\cdots(1-X_{2n-1}^{2})}\underset{\geq}{\Omega}\frac{1}{1-X_{2n}^2\lambda_{2n}^2}\\
=&\frac{(1+X_1)(1+X_2X_3)\cdots(1+X_{2n-2}X_{2n-1})}{(1-X_1^2)(1-X_2^2)\cdots(1-X_{2n}^{2})}.
\end{align*}
So, we have shown that
$$\sum_{\substack{\alpha\in\mathcal{P}_1'\\\ell(\lambda)\leq2n}}x^{\alpha}=\frac{(1+X_1)(1+X_2X_3)\cdots(1+X_{2n-2}X_{2n-1})}{(1-X_1^2)(1-X_2^2)\cdots(1-X_{2n}^{2})}.$$
Let $n\to\infty$, we get
$$\sum_{\lambda\in\mathcal{P}_1'}\lambda^x=\frac{1+X_1}{1-X_1^2}\cdot\frac{1}{1-X_2^2}\cdot\frac{1+X_2X_3}{1-X_3^2}\cdot\frac{1}{1-X_4^2}\cdots,$$
so we finished the proof.
\end{proof}
Now we present the identity related to $\mathcal{P}_2'$, which is the set of  partitions with odd-indexed parts being even.
\begin{theorem}\label{LambdaXP2'}
For the partitions in $\mathcal{P}'_2$,
\begin{align*}
\sum_{\lambda\in\mathcal{P}'_2}x^{\lambda}=&1+\sum_{n=1}^{\infty}\frac{X_{2n}^2(1+X_1X_2)(1+X_3X_4)\cdots(1+X_{2n-3}X_{2n-2})(1+X_{2n-1}/X_{2n})}{(1-X_1^2)(1-X_2^2)\cdots(1-X_{2n}^2)}\\
&+\sum_{n=0}^{\infty}\frac{X_{2n+1}^2(1+X_1X_2)(1+X_3X_4)\cdots(1+X_{2n-1}X_{2n})}{(1-X_1^2)(1-X_2^2)\cdots(1-X_{2n+1}^2)}\\
=&\frac{1}{1-X_1^2}\cdot\frac{1+X_1X_2}{1-X_2^2}\cdot\frac{1}{1-X_3^2}\cdot\frac{1+X_3X_4}{1-X_4^2}\cdots.
\end{align*}
\end{theorem}
\begin{proof}
Similar to the proof of Theorem \ref{LambdaXP1'}, we will present the series form and product form separately. We start by considering such partitions with fixed length $2n$ or $2n+1$, which leads to the series form. For those with even length, the crude form is given by
\begin{align*}
\sum_{\substack{\alpha\in\mathcal{P}_2'\\\ell(\lambda)=2n}}x^{\alpha}=&
\underset{\geq}{\Omega}\sum_{a_1,\ldots a_{2n}\geq0}x_1^{2a_1}x_2^{a_{2}}\cdots x_{2n-1}^{2a_{2n-1}}x_{2n}^{a_{2n}}\lambda_{1}^{2a_1-a_{2}}\lambda_{2}^{a_2-2a_3}\cdots\lambda_{2n-1}^{2a_{2n-1}-a_{2n}}\lambda_{2n}^{a_{2n}-1}\\
=&\underset{\geq}{\Omega}\frac{1}{1-x_1^2\lambda_1^2}\prod_{i=2}^{n}\left(1-\frac{x_{2i-1}^2\lambda_{2i-1}^2}{\lambda_{2i-2}^2}\right)^{-1}\prod_{i=1}^{n}\left(1-\frac{x_{2i}\lambda_{2i}}{\lambda_{2i-1}}\right)^{-1}\cdot\frac{1}{\lambda_{2n}},
\end{align*}
and by \eqref{EliminationRule4} and \eqref{EliminationRule6}, we get
\begin{align*}
&\underset{\geq}{\Omega}\frac{1}{1-x_1^2\lambda_1^2}\prod_{i=2}^{n}\left(1-\frac{x_{2i-1}^2\lambda_{2i-1}^2}{\lambda_{2i-2}^2}\right)^{-1}\prod_{i=1}^{n}\left(1-\frac{x_{2i}\lambda_{2i}}{\lambda_{2i-1}}\right)^{-1}\cdot\frac{1}{\lambda_{2n}}\\
=&\frac{1}{1-X_1^2}\underset{\geq}{\Omega}\frac{1+X_1X_2\lambda_2}{1-X_2^2\lambda_2^2}\prod_{i=2}^{n}\left(1-\frac{x_{2i-1}^2\lambda_{2i-1}^2}{\lambda_{2i-2}^2}\right)^{-1}\prod_{i=2}^{n}\left(1-\frac{x_{2i}\lambda_{2i}}{\lambda_{2i-1}}\right)^{-1}\cdot\frac{1}{\lambda_{2n}}\\
=&\frac{1+X_1X_2}{(1-X_1^2)(1-X_2^2)}\underset{\geq}{\Omega}\frac{1}{1-X_3^2\lambda_3^2}\prod_{i=3}^{n}\left(1-\frac{x_{2i-1}^2\lambda_{2i-1}^2}{\lambda_{2i-2}^2}\right)^{-1}\prod_{i=2}^{n}\left(1-\frac{x_{2i}\lambda_{2i}}{\lambda_{2i-1}}\right)^{-1}\cdot\frac{1}{\lambda_{2n}}\\
&\vdots\\
=&\frac{(1+X_1X_2)(1+X_3X_4)\cdots(1+X_{2n-3}X_{2n-2})}{(1-X_1^2)(1-X_2^2)\cdots(1-X_{2n-1}^{2})}\underset{\geq}{\Omega}\frac{1+X_{2n-1}X_{2n}\lambda_{2n}}{(1-X_{2n}^2\lambda_{2n}^2)\lambda_{2n}}\\
=&\frac{X_{2n}^{2}(1+X_1X_2)(1+X_3X_4)\cdots(1+X_{2n-3}X_{2n-2})(1+X_{2n}/X_{2n-1})}{(1-X_1^2)(1-X_2^2)\cdots(1-X_{2n}^{2})}.   
\end{align*}
We have shown that
$$\sum_{\substack{\alpha\in\mathcal{P}_2'\\\ell(\alpha)=2n}}x^{\alpha}
=\frac{X_{2n}^{2}(1+X_1X_2)(1+X_3X_4)\cdots(1+X_{2n-3}X_{2n-2})(1+X_{2n}/X_{2n-1})}{(1-X_1^2)(1-X_2^2)\cdots(1-X_{2n}^{2})},$$
with a similar argument, we get
$$\sum_{\substack{\alpha\in\mathcal{P}_2'\\\ell(\alpha)=2n+1}}x^{\alpha}
=\frac{X_{2n+1}^{2}(1+X_1X_2)(1+X_3X_4)\cdots(1+X_{2n-1}X_{2n})}{(1-X_1^2)(1-X_2^2)\cdots(1-X_{2n+1}^{2})}.$$
Adding up the all the nonnegative integers $n$, we have
\begin{align*}
\sum_{\lambda\in\mathcal{P}'_2}x^{\lambda}=&1+\sum_{n=1}^{\infty}\frac{X_{2n}^{2}(1+X_1X_2)(1+X_3X_4)\cdots(1+X_{2n-3}X_{2n-2})(1+X_{2n}/X_{2n-1})}{(1-X_1^2)(1-X_2^2)\cdots(1-X_{2n}^{2})}\\
&+\sum_{n=0}^{\infty}\frac{X_{2n+1}^{2}(1+X_1X_2)(1+X_3X_4)\cdots(1+X_{2n-1}X_{2n})}{(1-X_1^2)(1-X_2^2)\cdots(1-X_{2n+1}^{2})},    
\end{align*}
so we have the series form as desired. For the product form, it suffices to consider such partitions with length bounded by $2n$, which has the crude form as
\begin{align*}
\sum_{\substack{\alpha\in\mathcal{P}_2'\\\ell(\lambda)\leq2n}}x^{\alpha}=&
\underset{\geq}{\Omega}\sum_{a_1,\ldots a_{2n}\geq0}x_1^{2a_1}x_2^{a_{2}}\cdots x_{2n-1}^{2a_{2n-1}}x_{2n}^{a_{2n}}\lambda_{1}^{2a_1-a_{2}}\lambda_{2}^{a_2-2a_3}\cdots\lambda_{2n-1}^{2a_{2n-1}-a_{2n}}\lambda_{2n}^{a_{2n}}\\
=&\underset{\geq}{\Omega}\frac{1}{1-x_1^2\lambda_1^2}\prod_{i=2}^{n}\left(1-\frac{x_{2i-1}^2\lambda_{2i-1}^2}{\lambda_{2i-2}^2}\right)^{-1}\prod_{i=1}^{n}\left(1-\frac{x_{2i}\lambda_{2i}}{\lambda_{2i-1}}\right)^{-1}.
\end{align*}
By the same elimination process, we have
\begin{align*}
&\underset{\geq}{\Omega}\frac{1}{1-x_1^2\lambda_1^2}\prod_{i=2}^{n}\left(1-\frac{x_{2i-1}^2\lambda_{2i-1}^2}{\lambda_{2i-2}^2}\right)^{-1}\prod_{i=1}^{n}\left(1-\frac{x_{2i}\lambda_{2i}}{\lambda_{2i-1}}\right)^{-1}\\
=&\frac{1}{1-X_1^2}\underset{\geq}{\Omega}\frac{1+X_1X_2\lambda_2}{1-X_2^2\lambda_2^2}\prod_{i=2}^{n}\left(1-\frac{x_{2i-1}^2\lambda_{2i-1}^2}{\lambda_{2i-2}^2}\right)^{-1}\prod_{i=2}^{n}\left(1-\frac{x_{2i}\lambda_{2i}}{\lambda_{2i-1}}\right)^{-1}\\
=&\frac{1+X_1X_2}{(1-X_1^2)(1-X_2^2)}\underset{\geq}{\Omega}\frac{1}{1-X_3^2\lambda_3^2}\prod_{i=3}^{n}\left(1-\frac{x_{2i-1}^2\lambda_{2i-1}^2}{\lambda_{2i-2}^2}\right)^{-1}\prod_{i=2}^{n}\left(1-\frac{x_{2i}\lambda_{2i}}{\lambda_{2i-1}}\right)^{-1}\\
&\vdots\\
=&\frac{(1+X_1X_2)(1+X_3X_4)\cdots(1+X_{2n-3}X_{2n-2})}{(1-X_1^2)(1-X_2^2)\cdots(1-X_{2n-1}^{2})}\underset{\geq}{\Omega}\frac{1+X_{2n-1}X_{2n}\lambda_{2n}}{(1-X_{2n}^2\lambda_{2n}^2)}\\
=&\frac{(1+X_1X_2)(1+X_3X_4)\cdots(1+X_{2n-3}X_{2n-2})(1+X_{2n-1}X_{2n})}{(1-X_1^2)(1-X_2^2)\cdots(1-X_{2n}^{2})}.   
\end{align*}
So we have shown that
$$\sum_{\substack{\alpha\in\mathcal{P}_2'\\\ell(\alpha)\leq2n}}x^{\alpha}
=\frac{(1+X_1X_2)(1+X_3X_4)\cdots(1+X_{2n-1}X_{2n})}{(1-X_1^2)(1-X_2^2)\cdots(1-X_{2n}^{2})}.$$
By letting $n\to\infty$, we will have the desired product form, hence we finish the proof.
\end{proof}
\section{Applications of the refined generating functions}\label{sec:Application}
In this section, we give some applications of the refined generating functions. We will need the following identities. The first one is a $q$-analogue of Gauss's identity~\cite{qGauss}.
\begin{theorem}[$q$-Gauss summation]\label{thm:qGauss}
For $|q|<1$ and $|c/ab|<1$,
\begin{equation}\label{eq:qGauss}
\sum_{n=0}^{\infty}\frac{(a;q)_{n}(b;q)_{n}}{(q;q)_{n}(c;q)_{n}}\left(\frac{c}{ab}\right)^{n}=\frac{(c/a;q)_{\infty}(c/b;q)_{\infty}}{(c;q)_{\infty}(c/ab;q)_{\infty}}.    
\end{equation}
And the following limit case holds by letting $a\to\infty$.
\begin{equation}\label{eq:qGaussLimit}
\sum_{n=0}^{\infty}\frac{(-1)^nq^{\binom{n}{2}}(b;q)_n}{(q;q)_n(c;q)_n}\left(\frac{c}{b}\right)^n=\frac{(c/b;q)_{\infty}}{(c;q)_{\infty}}.
\end{equation}
\end{theorem}
The second one was due to V. A. Lebesgue~\cite{Lebesgue}.
\begin{theorem}[$q$-Lebesgue identity]\label{thm:qLebesgue}
For $|q|<1$,
\begin{equation}\label{eq:qLebesgue}
\sum_{n=0}^{\infty}\frac{(-aq;q)_{n}}{(q;q)_{n}}q^{\binom{n+1}{2}}=(-aq;q^{2})_{\infty}(-q;q)_{\infty}=\frac{(-aq;q^2)_{\infty}}{(q;q^2)_{\infty}}.   
\end{equation}
\end{theorem}
We will mainly be interested in the alternating sum and the Schmidt weight of partitions. Given a partition $\lambda=(\lambda_1,\lambda_2,\ldots,\lambda_{\ell})$, the alternating sum of $\lambda$ is defined as
$$a(\lambda)=\lambda_1-\lambda_2+\lambda_3-\lambda_4+\cdots+(-1)^{\ell-1}\lambda_{l}.$$
This was first introduced by M. Bousquet-Mélou and K. Eriksson~\cite{Bousquet-MélouEriksson}. In their paper, they considered the difference between odd-indexed parts and even-indexed parts. By a combinatorial argument, the alternating sum of a partition is equivalent to the number of odd parts in the conjugate of that partition.

It is worth to mention that Euler's partition theorem, which states that the number of partitions of $n$ in to odd parts is equal to the number of partitions of $n$ into distinct parts, has a nice refinement involving the alternating sum~\cite{Sylvester}.
\begin{theorem}[Sylvester] For any nonnegative integers $m$ and $n$, the number of partitions of $n$ into $m$ odd parts is equal to the number of partitions of $n$ into distinct parts with alternating sum $m$.
\end{theorem}
This was proved by his fish-hook bijection, and some further investigation can be found in \cite{WangZhengEuler}. Here are some more studies related to alternating sums, see~\cite{AndrewsAlternatingSum,BerkovichUncu,KangLiWang} for examples. 

The Schmidt weight of a partition is defined as
$$S(\lambda):=\lambda_1+\lambda_{3}+\lambda_{5}+\cdots,$$
which is the sum of all the odd-indexed parts in the partition. This weight was motivated by a conjecture posted by Schmidt in 1999 and was recently revisited by George and Paule in \cite{AndrewsPauleSchmidt}. Since then, the Schmidt weight has received huge attention, and a number of papers have been devoted on this topic, see \cite{AlladiSchmidt,AndrewsKeithSchmidt,LiYee,WangZhengSchmidt} for examples. 

Both the alternating sum and the Schmidt weight require different variables to keep track of odd-indexed parts and even-indexed parts separately. This can be easily done by the refined weight $x^{\lambda}$. By setting $x_{2i-1}\to zq$ and $x_{2i}\to q/z$, which means $X_{2i-1}\to zq^{2i-1}$ and $X_{2i}\to q^{2i}$, for all $i$, we get the weight $z^{a(\lambda)}q^{|\lambda|}$. Similarly, by setting $x_{2i-1}\to zq$ and $x_{2i}\to 1/z$, which means $X_{2i-1}\to zq^{i}$ and $X_{2i}\to q^{i}$, we get the weight $z^{a(\lambda)}q^{S(\lambda)}$. So, the following partition identities come as consequences of Theorem~\ref{thm:main1} and \ref{thm:main2}.
\begin{theorem}\label{thm:BivariableG}
For partitions in $\mathcal{G}_1$ and $\mathcal{G}_2$,
\begin{equation}\label{G1AlternatingSum}
\sum_{\lambda\in\mathcal{G}_1}z^{a(\lambda)}q^{|\lambda|}=\sum_{n=0}^{\infty}\frac{z^{2n}q^{2\binom{2n+1}{2}}(-z/q;q^2)_{2n}}{(q^4;q^4)_{n}(z^2q^2;q^4)_{n}}+\sum_{n=0}^{\infty}\frac{z^{2n+2}q^{2\binom{2n+2}{2}}(-zq;q^2)_{2n}(1+1/zq)}{(q^4;q^4)_{n}(z^2q^2;q^4)_{n+1}},    
\end{equation}
\begin{equation}\label{G2AlternatingSum}
\sum_{\lambda\in\mathcal{G}_2}z^{a(\lambda)}q^{|\lambda|}=\sum_{n=0}^{\infty}\frac{z^{2n}q^{2\binom{2n+1}{2}}(-zq;q^2)_{2n}}{(q^4;q^4)_{n}(z^2q^2;q^4)_{n}}+\sum_{n=0}^{\infty}\frac{z^{2n+2}q^{2\binom{2n+2}{2}}(-zq;q^2)_{2n+1}}{(q^4;q^4)_{n}(z^2q^2;q^4)_{n+1}},     
\end{equation}
\begin{equation}
\sum_{\lambda\in\mathcal{G}_1}z^{a(\lambda)}q^{S(\lambda)}=\sum_{n=0}^{\infty}\frac{z^{2n}q^{4\binom{n+1}{2}}(-z;q)_{2n}}{(q^2;q^2)_{n}(z^2q^2;q^2)_{n}}+\sum_{n=0}^{\infty}\frac{z^{2n+2}q^{2(n+1)^2}(-zq;q)_{2n}(1+1/z)}{(q^2;q^2)_{n}(z^2q^2;q^2)_{n+1}},      
\end{equation}
\begin{equation}
\sum_{\lambda\in\mathcal{G}_2}z^{a(\lambda)}q^{S(\lambda)}=\sum_{n=0}^{\infty}\frac{z^{2n}q^{4\binom{n+1}{2}}(-zq;q)_{2n}}{(q^2;q^2)_{n}(z^2q^2;q^2)_{n}}+\sum_{n=0}^{\infty}\frac{z^{2n+2}q^{2(n+1)^2}(-zq;q)_{2n+1}}{(q^2;q^2)_{n}(z^2q^2;q^2)_{n+1}}.    
\end{equation}
\end{theorem}
\begin{proof}
These are straightforward from Theorem \ref{LambdaXG1} and Theorem \ref{LambdaXG2} with proper substitutions as mentioned above.
\end{proof}
\begin{theorem}\label{thm:BivariableG'}
For partitions in $\mathcal{G}_1'$ and $\mathcal{G}_2'$,
\begin{equation}\label{G1'AlternatingSum}
\sum_{\lambda\in\mathcal{G}_1'}z^{a(\lambda)}q^{|\lambda|}=\sum_{n=0}^{\infty}\frac{z^nq^{\binom{2n}{2}}(-zq;q^4)_{n}}{(q^4;q^4)_{n}(z^2q^2;q^4)_{n}}=\frac{(-zq;q^4)_{\infty}}{(z^2q^2;q^4)_{\infty}}=\frac{1}{(zq;q^4)_{\infty}(z^2q^6;q^8)_{\infty}},    
\end{equation}
\begin{equation}\label{G2'AlternatingSum}
\sum_{\lambda\in\mathcal{G}_2'}z^{a(\lambda)}q^{|\lambda|}=\sum_{n=0}^{\infty}\frac{z^{n}q^{\binom{2n+1}{2}}(-zq^{-1};q^4)_{n}}{(q^4;q^4)_{n}(z^2q^2;q^4)_{n}}=\frac{(-zq^3;q^4)_{\infty}}{(z^2q^2;q^4)_{\infty}}=\frac{1}{(zq^3;q^4)_{\infty}(z^2q^2;q^8)_{\infty}},    
\end{equation}
\begin{equation}
\sum_{\lambda\in\mathcal{G}_1'}z^{a(\lambda)}q^{S(\lambda)}=\sum_{n=0}^{\infty}\frac{z^nq^{n^2}(-zq;q^2)_n}{(q^2;q^2)_n(z^2q^2;q^2)_n}=\frac{(-zq;q^2)_{\infty}}{(z^2q^2;q^2)_{\infty}}=\frac{1}{(zq;q^2)_{\infty}(z^2q^4;q^4)_{\infty}},     
\end{equation}
\begin{equation}
\sum_{\lambda\in\mathcal{G}_2'}z^{a(\lambda)}q^{S(\lambda)}=\sum_{n=0}^{\infty}\frac{z^nq^{n^2+n}(-zq^{-1};q^2)_{n}}{(q^2;q^2)_n(z^2q^2;q^2)_{n}}=\frac{(-zq^2;q^2)_{\infty}}{(z^2q^2;q^2)_{\infty}}=\frac{1}{(zq^2;q^2)_{\infty}(z^2q^2;q^4)_{\infty}}.    
\end{equation}    
\end{theorem}
\begin{proof}
The summation sides follow directly from the substitution of variables in Theorem \ref{LambdaXG1'} and Theorem \ref{LambdaXG2'}, while the product sides follow from the application of \eqref{eq:qGaussLimit} to the summation sides.  
\end{proof}
\begin{remark}
Note that for Theorem~\ref{thm:BivariableG}, we do not have the bi-variable product forms corresponding to them. Here is the reason. The little Göllnitz identities are special cases of \eqref{eq:qLebesgue}. But this identity can no longer be applied to the bi-variable series.
\end{remark}
\begin{theorem}\label{the:BivariableP}
For partitions in $\mathcal{P}_1$ and $\mathcal{P}_2$, we have
\begin{equation}\label{P1AlternatingSum}
\begin{split}
\sum_{\lambda\in\mathcal{P}_1}z^{a(\lambda)}q^{|\lambda|}=&1+\sum_{n=1}^{\infty}\frac{q^{4n}(-zq;q^8)_{n}(-z^3q^5;q^8)_{n-1}(1+z^3q^{4n-3})}{(q^4;q^4)_{n}(z^2q^2;q^4)_{n}}\\
&+\sum_{n=0}^{\infty}\frac{z^2q^{4n+2}(-zq;q^8)_{n}(-z^3q^5;q^8)_{n}(1+z^{-1}q^{4n-1})}{(q^4;q^4)_{n}(z^2q^2;q^4)_{n+1}}\\
=&\frac{(-zq;q^8)_{\infty}(-z^3q^5;q^8)_{\infty}}{(z^2q^2;q^4)_{\infty}(q^4;q^4)_{\infty}},    
\end{split}   
\end{equation}
\begin{equation}\label{P21AlternatingSum}
\begin{split}
\sum_{\lambda\in\mathcal{P}_2}z^{a(\lambda)}q^{|\lambda|}=&1+\sum_{n=1}^{\infty}\frac{q^{4n}(-z^3q^3;q^8)_{n}(-zq^7;q^8)_{n-1}(1+zq^{4n-1})}{(q^4;q^4)_{n}(z^2q^2;q^4)_{n}}\\
&+\sum_{n=0}^{\infty}\frac{z^2q^{4n+2}(-z^3q^3;q^8)_{n}(-zq^7;q^8)_{n}(1+zq^{4n+1})}{(q^4;q^4)_{n}(z^2q^2;q^4)_{n+1}}\\
=&\frac{(-z^3q^3;q^8)_{\infty}(-zq^7;q^8)_{\infty}}{(z^2q^2;q^4)_{\infty}(q^4;q^4)_{\infty}},    
\end{split}     
\end{equation}
\begin{equation}
\begin{split}
\sum_{\lambda\in\mathcal{P}_1}z^{a(\lambda)}q^{S(\lambda)}=&1+\sum_{n=1}^{\infty}\frac{q^{2n}(-zq;q^4)_{n}(-z^3q^4;q^4)_{n-1}(1+z^3q^{2n})}{(q^2;q^2)_{n}(z^2q^2;q^2)_{n}}\\
&+\sum_{n=0}^{\infty}\frac{z^2q^{2n+2}(-zq;q^4)_{n}(-z^3q^4;q^4)_{n}(1+z^{-1}q^{2n})}{(q^2;q^2)_{n}(z^2q^2;q^2)_{n+1}}\\
=&\frac{(-zq;q^4)_{\infty}(-z^3q^4;q^4)_{\infty}}{(z^2q^2;q^2)_{\infty}(q^2;q^2)_{\infty}},    
\end{split}   
\end{equation}
\begin{equation}
\begin{split}
\sum_{\lambda\in\mathcal{P}_2}z^{a(\lambda)}q^{S(\lambda)}=&1+\sum_{n=1}^{\infty}\frac{q^{2n}(-z^3q^3;q^4)_{n}(-zq^4;q^4)_{n-1}(1+zq^{2n})}{(q^2;q^2)_{n}(z^2q^2;q^2)_{n}}\\
&+\sum_{n=0}^{\infty}\frac{z^2q^{2n+2}(-z^3q^3;q^4)_{n}(-zq^4;q^4)_{n}(1+zq^{2n+1})}{(q^2;q^2)_{n}(z^2q^2;q^2)_{n+1}}\\
=&\frac{(-z^3q^3;q^4)_{\infty}(-zq^4;q^4)_{\infty}}{(z^2q^2;q^2)_{\infty}(q^2;q^2)_{\infty}}.  
\end{split}     
\end{equation}  
\end{theorem}
\begin{proof}
Both the summation side and product side follow from the substitution of $X_i$'s in Theorem \ref{LambdaXP1} and Theorem \ref{LambdaXP2}.
\end{proof}
\begin{theorem}\label{thm:BivariableP'}
For partitions in $\mathcal{P}_1'$ and $\mathcal{P}_2'$,
\begin{equation}\label{P1'AlternatingSum}
\begin{split}
\sum_{\lambda\in\mathcal{P}_1'}z^{a(\lambda)}q^{|\lambda|}=&\sum_{n=0}^{\infty}\frac{q^{4n}(-zq^{-3};q^4)_{n}}{(q^4;q^4)_{n}(z^2q^2;q^4)_{n}}+\sum_{n=0}^{\infty}\frac{z^2q^{4n+2}(-zq;q^4)_{n}}{(q^4;q^4)_{n}(z^2q^2;q^4)_{n+1}}\\
=&\frac{(-zq;q^4)_{\infty}}{(z^2q^2;q^4)_{\infty}(q^4;q^4)_{\infty}},    
\end{split}   
\end{equation}
\begin{equation}\label{P2'AlternatingSum}
\begin{split}
\sum_{\lambda\in\mathcal{P}_2'}z^{a(\lambda)}q^{|\lambda|}=&\sum_{n=0}^{\infty}\frac{q^{4n}(-z/q;q^4)_{n}}{(q^4;q^4)_{n}(z^2q^2;q^4)_{n}}+\sum_{n=0}^{\infty}\frac{z^2q^{4n+2}(-zq^3;q^4)_{n}}{(q^4;q^4)_{n}(z^2q^2;q^4)_{n+1}}\\
=&\frac{(-zq^3;q^4)_{\infty}}{(z^2q^2;q^4)_{\infty}(q^4;q^4)_{\infty}},    
\end{split}     
\end{equation}
\begin{equation}
\begin{split}
\sum_{\lambda\in\mathcal{P}_1'}z^{a(\lambda)}q^{S(\lambda)}=&\sum_{n=0}^{\infty}\frac{q^{2n}(-zq^{-1};q^2)_{n}}{(q^2;q^2)_{n}(z^2q^2;q^2)_{n}}+\sum_{n=0}^{\infty}\frac{z^2q^{2n+2}(-zq;q^2)_{n}}{(q^2;q^2)_{n}(z^2q^2;q^2)_{n+1}}\\
=&\frac{(-zq;q^2)_{\infty}}{(z^2q^2;q^2)_{\infty}(q^2;q^2)_{\infty}},    
\end{split}   
\end{equation}
\begin{equation}
\begin{split}
\sum_{\lambda\in\mathcal{P}_2'}z^{a(\lambda)}q^{S(\lambda)}=&1+\sum_{n=1}^{\infty}\frac{q^{2n}(-z;q^2)_{n}}{(q^2;q^2)_{n}(z^2q^2;q^2)_{n}}+\sum_{n=0}^{\infty}\frac{z^2q^{2n+2}(-zq^2;q^2)_{n}}{(q^2;q^2)_{n}(z^2q^2;q^2)_{n+1}}\\
=&\frac{(-zq^2;q^2)_{\infty}}{(z^2q^2;q^2)_{\infty}(q^2;q^2)_{\infty}}.  
\end{split}     
\end{equation} 
\end{theorem}
\begin{proof}
These identities are consequences of Theorem \ref{LambdaXP1'} and Theorem \ref{LambdaXP2'} with proper substitutions.
\end{proof}

\section{Conclusion}\label{sec:Conclusion}
This project is largely motivated by \cite{AndrewsPauleParity} where the authors studied partitions corresponding to the sum sides of the Göllnitz-Gordan identities. Here we propose some problems that might lead to further study.
\begin{enumerate}
    \item It is natural to generalize Andrews' mod $8$ identity and its companions. For any two integers $k>i\geq0$, a part $\lambda_m$ in a partition is called $(k,i)$-indexed if $m\equiv i \pmod{k}$. Let $\mathcal{P}_{(t_1,t_2,\ldots,t_i;k)}$ be the set of partitions such that for all $1\leq j\leq i$ the $(t_j,k)$-indexed parts must be even and the odd parts must be distinct, then the computation suggests the following.
    \begin{conjecture}
    \begin{equation}
    \sum_{\lambda\in\mathcal{P}_{(t_1,t_2,\ldots,t_i;k)}}x^{\lambda}=\prod_{n=1}^{\infty}\frac{1}{1-X_n^2}\prod_{n\not\equiv t_1,\ldots t_i\pmod{k}}(1+X_{n-1}X_n).
    \end{equation}
    \end{conjecture}
    It's easy to see that $\mathcal{P}_1'$ and $\mathcal{P}_2'$ corresponds to the special case $\mathcal{P}_{(0;2)}$ and $\mathcal{P}_{(1;2)}$ respectively. Is there similar generalization for $\mathcal{G}_1'$ or $\mathcal{G}_2'$?
    \item In this paper, we considered partitions with certain prescribed position being even. Will there be any interesting identity for partitions with certain prescribed position being odd?
    \item For a partition set $\mathcal{S}$, let $\mathcal{S}(n)$ be the subset of $\mathcal{S}$ consist of all partitions with weight $n$. The identities we present indicate that $|\mathcal{G}_1(n)|=|\mathcal{G}_1'(n)|$ and $|\mathcal{G}_2(n)|=|\mathcal{G}_2'(n)|$. Since the proof of these identities relies on the $q$-Lebesgue identity and $q$-Gauss identity, for which a combinatorial proof can be found in \cite{Fu,Yee}, the composition of these maps would provide a combinatorial proof between the $\mathcal{G}_i$ and $\mathcal{G}_i'$. However, it would also be nice to have a direct bijection between them, without invoking partitions counted by the product sides.
    \item In~\cite{CraigLi}, the authors studied the distribution of alternating sums over all partitions. It would also be interesting to study asymptotic properties for the partition functions in Section~\ref{sec:Application}.
\end{enumerate}


\begin{thebibliography}{99}

\bibitem{AlladiSchmidt}
K. Alladi, \emph{Schmidt-type theorems via weighted partition identities}, Ramanujan J., \textbf{61} (2023), 701--714.

\bibitem{AndrewsLectureHall}
G.E. Andrews, \emph{ MacMahon's partition analysis I: The lecture hall partition theorem},  Mathematical essays in honor of Gian-Carlo Rota (Cambridge, MA, 1996), 1--22.

\bibitem{AndrewsAlternatingSum}
G. E. Andrews, \emph{On a partition function of Richard Stanley}, Electron. J. Combin., \textbf{11} (2004/06), no. 2, Research Paper 1, 10 pp.

\bibitem{AndrewsSchur}
G. E. Andrews, \emph{On Schur’s second partition theorem}, Glasgow Math. J., \textbf{8} (1967), 127--132, DOI 10.1017/S0017089500000197. MR220692

\bibitem{AndrewsSIP}
G. E. Andrews, \emph{Separable integer partition classes}, Trans. Amer. Math. Soc. Ser. B, \textbf{9} (2022), 619--647.

\bibitem{AndrewsKeithSchmidt}
G. E. Andrews and W. Keith, \emph{A general class of Schmidt theorems}, J. Number Theory, \textbf{247} (2023), 75--99.


\bibitem{AndrewsPauleSchmidt}
G. E. Andrews and P. Paule, \emph{MacMahon's partition analysis XIII: Schmidt type partitions and modular forms}, J. Number Theory, \textbf{234} (2022), 95--119.

\bibitem{AndrewsPauleNN}
G. E. Andrews and P. Paule, \emph{MacMahon's partition analysis XIV: Partitions with $n$ copies of $n$}, J. Combin. Theory Ser. A, \textbf{203} (2024), 105836.

\bibitem{AndrewsPauleParity}
G. E. Andrews and P. Paule, \emph{MacMahon's partition analysis XV: Parity}, J. Symbolic Comput., \textbf{127} (2025), 102351.

\bibitem{OmegaPackage} 
G.E. Andrews, P. Paule, and A. Riese, \emph{MacMahon’s Partition Analysis III: The Omega Package}, European J. Combin., \textbf{22} (2001), 887--904.

\bibitem{BerkovichUncu}
A. Berkovich and A. K. Uncu, \emph{On partitions with fixed number of even-indexed and odd-indexed odd parts}, J. Number Theory, {\bf 167} (2016), 7--30.

\bibitem{Bousquet-MélouEriksson}
M. Bousquet-Mélou and K. Eriksson, {\it Lecture Hall Partitions}, Ramanujan J., \textbf{1} (1997), 101--111. 

\bibitem{CraigLi} W. Craig and R. Li, {\it Distribution of Alternating Sums of Parts in Partitions}, Ramanujan J., Accepted.

\bibitem{Fu}
A. M. Fu, \emph{A combinatorial proof of the Lebesgue identity}, Discrete Math., \textbf{308} (2008), 2611--2613.

\bibitem{Göllnitz}
H. Göllnitz, {\it Partitionen mit Differenzenbedingungen}, J. Reine Angew. Math., \textbf{225} (1967), 154--190.

\bibitem{qGauss}
B. Gordon and R. J. McIntosh,{\it Algebraic Dilogarithm Identities}, Ramanujan J., \textbf{1} (1997), 431--448.

\bibitem{KangLiWang}
J. Kang, R. Li and A. Y. Z. Wang, \emph{Partition identities related to the minimal excludant}, Discrete Math., {\bf 346} (2023), 113302.

\bibitem{Lebesgue}
V. A. Lebesgue, {\it Sommation de quelques séries}, J. Math. Pures Appl., \textbf{5} (1840), 42--71.


\bibitem{LiYee}
R. Li and A. J. Yee, \emph{Schmidt type partitions}, Enumer. Comb. Appl., \textbf{3} (2023),
S2R3.


\bibitem{MacMahon}
P. A. MacMahon, \emph{Combinatorial analysis}, Volumes I and II, Vol. 137. American Mathematical Society, 2001.

\bibitem{OmegaWeb} 
A. Riese, \texttt{Omega} \textit{- A Mathematica Implementation of Partition Analysis} \href{https://risc.jku.at/sw/omega/} Accessed: Dec 3, 2025.

\bibitem{SavageSills}
S. D. Savage and A. V. Sills, \emph{On an identity of Gessel and Stanton and the new little Göllnitz identities}, Advances in Applied Mathematics, \textbf{46} (2011), 563--575.

\bibitem{Sylvester}
J. Sylvester, {\it  A constructive theory of partitions, arranged in three acts, an interact and
an exodion}, Amer. J. Math., \textbf{5} (1882), 251--330.

\bibitem{WangZhengEuler}
A. Y. Z. Wang and Z. Xu, \emph{The minimal odd excludant and Euler's partition theorem}, Int. J. Number Theory, {\bf 20(6)} (2024), 1445--1461.

\bibitem{WangZhengSchmidt}
A. Y. Z. Wang and Z. Xu, \emph{The minimal excludant and Schmidt's partition theorem}, Discrete Math. Discrete Math., \textbf{346} (2023), no. 5, Paper No. 113309, 10 pp.

\bibitem{Yee}
A. J. Yee, \emph{Combinatorial proofs of Ramanujan's $_1\psi_1$ summation and $q$-Gauss summation}, J. Combin. Theory Ser. A, \textbf{105} (2004), 63--77.
\end{thebibliography}
\end{document}